\documentclass[a4paper,11pt]{article}
\usepackage{graphicx}
\usepackage{placeins}
\usepackage[english,activeacute]{babel}
\usepackage[latin1]{inputenc}
\usepackage{amsmath}
\usepackage{amsfonts}
\usepackage{amssymb}
\usepackage{amsthm}
\usepackage{latexsym}
\usepackage{hyperref}
\usepackage{xcolor}
\usepackage{booktabs}
\usepackage{orcidlink}
\usepackage{mathtools}
\usepackage{comment}
\hypersetup{colorlinks=true,linkcolor=blue!55!black,citecolor=green!45!black,%
            urlcolor=blue!55!black}

\newtheorem{theorem}{Theorem}[section]
\newtheorem{proposition}[theorem]{Proposition}
\newtheorem{lemma}[theorem]{Lemma}
\newtheorem{corollary}[theorem]{Corollary}

\newtheorem{remark}[theorem]{Remark}

\newcommand{\Z}{\mathbb{Z}}
\newcommand{\R}{\mathbb{R}}
\newcommand{\C}{\mathbb{C}}
\newcommand{\PP}{\mathbb{P}}
\newcommand{\Nzero}{\mathbb{N}_{0}}

\newcommand{\Kr}[1]{K^{p,N}_{#1}}                 % monic Krawtchouk
\newcommand{\KrM}[2]{K^{p,#1}_{#2}}               % monic Krawtchouk, shifted second parameter
\newcommand{\Sob}[1]{\mathbb{S}_{#1}}             % Sobolev-type family
\newcommand{\dif}{\Delta}                         % forward difference
\newcommand{\nab}{\nabla}                         % backward difference
\newcommand{\ff}[2]{\left[#1\right]_{#2}}         % falling factorial
\newcommand{\pFq}[2]{{}_{#1}F_{#2}}

\providecommand{\dst}{\displaystyle}
\newcommand{\Ker}{\mathcal{K}}

\allowdisplaybreaks

\begin{document}
\pagestyle{plain}

\title{A reciprocal-Gamma Mehler--Heine limit for Krawtchouk--Sobolev type
orthogonal polynomials with an exterior mass point}
\author{{Anier Soria-Lorente \orcidlink{0000-0003-3488-3094}}$^{1}$,
{Junior Michel \orcidlink{0009-0005-8450-2757}}$^{1}$,
{Junot Cacoq \orcidlink{0009-0000-9174-5312}}$^{2}$ \\
%EndAName
\\
$^{1}$Department of Quantitative Methods, Loyola University,\\ Avda. de
las Universidades, 2. Dos Hermanas, Seville,\\ 41704, Andalusia, Spain.\\
asoria@uloyola.es, jmichel@uloyola.es\\
$^{2}$Universit\'e Quisqueya, Port-au-Prince, Haiti\\
junot.cacoq@uniq.edu
}

\maketitle
\begin{abstract}
We study the monic polynomials $\Sob{n}$ orthogonal with respect to a Sobolev-type
modification of the discrete Krawtchouk inner product in which a single point mass at an
exterior point $\alpha<0$, outside the support $\{0,1,\dots,N\}$ of the binomial
measure, acts through the $j$-th forward difference $\dif^{j}$ (here $\Nzero=\{0,1,2,\dots\}$,
$\Z_{+}=\{1,2,\dots\}$, $\R_{+}=(0,\infty)$). In the joint scaling $n,N\to\infty$,
$n/N\to r$ with $0<p<r<1$ (the \emph{reciprocal-Gamma} regime of Dominici's
classification), we prove the Mehler--Heine formula
\begin{equation*}
	\lim_{\substack{n,N\to\infty\\ n/N\to r}}
	\frac{n}{\kappa_{n}\,\Gamma(n-z)}\,\Sob{n}(z)
	=-C_{\ast}\,(z-\alpha)\,\varphi_{K}(z),
	\quad
	C_{\ast}=\frac{r^{2}(1-p)}{(r-p)^{2}},
\end{equation*}
pointwise for each fixed $z\in\C\setminus\Nzero$. The exterior mass survives the
large-degree limit exactly through the linear factor $z-\alpha$, with the explicit
universal constant $C_{\ast}$, and the limit is independent of both $\lambda>0$ and
$j\ge1$; the same conclusion holds in the Uvarov case $j=0$ (Corollary~\ref{cor:uvarov}).
The derivation is pointwise and self-contained: it uses only a rank-one connection
formula (Theorem~\ref{thm:connection}), the reproducing-kernel recursion, Dominici's
pointwise limit and the classical three-term recurrence; the normalised Sobolev
correction obeys an exact affine recursion at fixed $N$, controlled by a triangular
contraction argument over two backward windows with an explicit boundary estimate, and
the contraction ratio $\varrho<1$ is exactly the condition $r>p$. We also discuss the
scale of the theorem: in this normalisation the classical term is asymptotically
negligible, and the independence of $\lambda$ concerns each fixed $\lambda>0$, not the
continuity at $\lambda=0$. Multiprecision computations illustrate the convergence and
the exact value $C_{\ast}=2.8$ for $p=0.3$, $r=3/5$, $\alpha=-2.5$, $j=2$, $\lambda=5$.
\end{abstract}

\vspace{0.3cm}

\par

\par

\smallskip
{\footnotesize Corresponding author: Anier Soria-Lorente\par}

%%%%%%%%%%%%%%%%%%%%%%%%%%%%%%%%%%%%%%%%%%%%%%%%%%%%%%%%%%%%%%%%%%%%%%%%%%%%%%%%%%%%
\section{Introduction}
\label{sec:intro}

Orthogonal polynomials on a finite discrete lattice combine explicit hypergeometric
representations, three-term recurrence relations, second-order difference equations,
Christoffel--Darboux kernels and shift identities within a rigid analytic structure
\cite{Ismail05,KLS2010,nikiforov1991classical}. Among the classical discrete families,
the Krawtchouk polynomials are the discrete analogue of the Hermite polynomials, with
the binomial distribution as orthogonality measure, and appear throughout coding theory,
the Hamming scheme, probability and combinatorics \cite{Dominici2020,KLS2010}. Their
finite support $\{0,1,\dots,N\}$ makes them a natural testing ground for the stability
of classical properties under perturbation.

Mehler--Heine type formulas describe the large-degree behaviour of a polynomial sequence
in a fixed neighbourhood of an endpoint of the spectrum, and through the limiting
profile they retain fine information about the underlying lattice. For the
Krawtchouk polynomials the degree is bounded by $N$, so a genuine large-degree limit
requires $n$ and $N$ to grow jointly; in the regime $n/N\to r$ with $r\in(0,1)$ Dominici
established two qualitatively different Mehler--Heine formulas, separated by the sign of
$r-p$ \cite{Dominici2020}. The present paper is devoted to the regime
\begin{equation}\label{eq:regime}
	0<p<r<1,
\end{equation}
which is governed by expression $(80)$ of \cite[Cor.~3(ii)]{Dominici2020}: there the
scaled classical polynomials converge to the \emph{reciprocal-Gamma profile}
$\varphi_{K}$ of \eqref{eq:phiK}, an entire function, a nonvanishing multiple of
$1/\Gamma(-z)$. This is the Krawtchouk analogue of the classical
Meixner and Charlier limits, and the setting in which one asks how a perturbation of the
measure disturbs an already profile with an already rich lattice structure.

The perturbation studied here arises from the interaction of two operations. The first
is a Sobolev-type modification of the orthogonality by a single point mass placed
\emph{outside} the support of the Krawtchouk measure, at $\alpha<0$. The second is that
this mass acts not on the function value but on the $j$-th forward difference
$\dif^{j}f(\alpha)$. The resulting inner product is
\begin{equation}\label{eq:innerprod}
	\langle f,g\rangle_{\lambda}
	=\sum_{x=0}^{N} f(x)g(x)\,\rho(x)
	+\lambda\,\dif^{j}f(\alpha)\,\dif^{j}g(\alpha),
\end{equation}
with $\rho(x)=\binom{N}{x}p^{x}(1-p)^{N-x}$ the classical Krawtchouk weight, $\alpha<0$,
$0<p<1$, $\lambda\in\R_{+}$ and $j\in\Z_{+}$. It is not evident a priori whether the
exterior mass remains visible, is amplified, or disappears in the large-degree limit, nor
how the operator $\dif^{j}$ interacts with the reciprocal-Gamma profile. As we show, in
the regime \eqref{eq:regime} the mass leaves a sharp signature in the limit: the
limiting profile carries the exact factor $z-\alpha$, so that the position of the
exterior point is recovered in the large-degree limit, while the difference order $j$
leaves no trace. Determining this behaviour and pinning down the exact constant is the
core of the paper.

Point-mass and difference modifications of the classical discrete measures have a long
structural history. For the Meixner, Krawtchouk and Charlier systems, the Uvarov
modification by a Dirac mass at an endpoint was studied by \'Alvarez-Nodarse, Garc\'ia
and Marcell\'an \cite{alvareznodarse1995}; orthogonality with respect to inner products
involving differences was initiated by Bavinck \cite{B1995}, and Mehler--Heine formulas
for the $\Delta$-Meixner--Sobolev family were obtained in \cite{mbDMS}; general surveys
are \cite{maxu,mapepi}, within which the local asymptotics of Sobolev-type
perturbations of classical discrete families have been obtained in
\cite{Costas-2022,mbDMS,SoriaMichel2026}. The Krawtchouk--Sobolev type family with higher-order differences
at the endpoints $0$ and $N$ was analysed by Huertas, Lastra and Soria-Lorente
\cite{Huertas2022}, where the mass sits \emph{on} the support; here it is placed at the
exterior point $\alpha<0$, off the support, which is decisive for the asymptotics. That
the position of an exterior mass survives in the local large-degree limit was shown for
the discrete Charlier and Meixner Sobolev-type families in \cite{SoriaMichel2026}, where
the limiting profile acquires the factor $z-\alpha$; the reciprocal-Gamma Krawtchouk
regime studied here is the counterpart of that phenomenon when the classical profile is
the rich lattice function $\varphi_{K}$.

Our specific contribution is the Mehler--Heine formula of Theorem~\ref{thm:mainMH}: for
each fixed $z\in\C\setminus\Nzero$ the scaled Sobolev polynomials converge to
$-C_{\ast}(z-\alpha)\varphi_{K}(z)$, with the explicit constant
$C_{\ast}=r^{2}(1-p)/(r-p)^{2}$, independently of $\lambda$ and $j$, and the same limit
holds in the Uvarov case $j=0$ (Corollary~\ref{cor:uvarov}). Three features make this
result sharp. First, the constant: $C_{\ast}$ is exact, and it arises as the product of
the two factors $r/(r-p)$ (the turn-point geometry fixed by the consecutive-index
expansion) and $1/(1-\varrho)$ (the fixed point of the triangular contraction), both
elementary objects computed explicitly. Second, the universality: the limit does not see
the mass strength $\lambda>0$ nor the difference order $j\ge1$, and the same value is
obtained in the pure Uvarov case $j=0$; at the scale of the theorem the mass is
therefore genuinely a perturbation of the profile of $\varphi_{K}$, not of individual
coefficients. Third, the mechanism: for every fixed $N$ the rank-one connection formula
and the kernel recursion give the exact affine relation
$\Xi_{m,N}=Q_{m,m-1,N}+\mu_{m,N}\Xi_{m-1,N}$; the diagonal limit is obtained by
iterating this triangular recursion on a macroscopic backward window on which the
coefficients are uniformly contractive and by proving that the inherited boundary term
vanishes, with a terminal $o(N)$ window identifying the fixed point. The limiting factor
is $\varrho=1/L_{t}<1$, precisely when $r>p$. The only nonelementary analytic input is a
pointwise two-term expansion of the consecutive-index ratio at $z$ and $\alpha$
(Lemma~\ref{lem:ratioexp}), obtained from the three-term recurrence; it fixes both the
factor $z-\alpha$ and the constant $C_{\ast}$. This route does not require passage to
the limit under the weighted average, and it uses Dominici's formula
\cite[Cor.~3(ii)]{Dominici2020} in its original pointwise form. The structural material
of Section~\ref{sec:sobolev} consists of the rank-one connection formula, the only
structural result needed for the asymptotics. Section~\ref{sec:numerics} gives the
numerical illustration, which includes an independent check of the exact constant, and
Section~\ref{sec:conclu} the conclusions.

%%%%%%%%%%%%%%%%%%%%%%%%%%%%%%%%%%%%%%%%%%%%%%%%%%%%%%%%%%%%%%%%%%%%%%%%%%%%%%%%%%%%
\section{Preliminaries on Krawtchouk polynomials}
\label{sec:prelim}

Throughout, $\Nzero=\{0,1,2,\dots\}$ denotes the nonnegative integers,
$\Z_{+}=\{1,2,\dots\}$ the positive integers and $\R_{+}=(0,\infty)$ the positive reals;
in particular $\lambda\in\R_{+}$ means $\lambda>0$, and $j\in\Z_{+}$ means $j\ge1$.
The Uvarov case $j=0$ is treated separately when it arises
(Section~\ref{subsec:uvarov}). We write $\dif$ and $\nab$ for the forward and backward
difference operators, $\dif f(x)=f(x+1)-f(x)$ and $\nab f(x)=f(x)-f(x-1)$, with
higher-order differences defined recursively and $\dif^{0}=\nab^{0}=\mathrm{Id}$. For
$k\in\Nzero$, $\ff{x}{k}$ denotes the falling factorial, $\ff{x}{0}=1$ and
$\ff{x}{k}=\prod_{i=0}^{k-1}(x-i)$ for $k\ge1$, and $(x)_{k}=\Gamma(x+k)/\Gamma(x)$ the
Pochhammer symbol. We write $\PP$ for the space of real polynomials.

\subsection{The classical monic Krawtchouk family}
\label{subsec:classicalK}

Let $0<p<1$ and $N\in\Z_{+}$. The monic Krawtchouk polynomials $\{\Kr{n}\}_{0\le n\le N}$
are orthogonal with respect to $\langle f,g\rangle=\sum_{x=0}^{N}f(x)g(x)\rho(x)$, with
the binomial weight
\begin{equation}\label{eq:weight}
	\rho(x)=\binom{N}{x}p^{x}(1-p)^{N-x}
	=\frac{\Gamma(N+1)\,p^{x}(1-p)^{N-x}}{\Gamma(N-x+1)\,\Gamma(x+1)},
	\quad 0\le x\le N,
\end{equation}
and admit the terminating hypergeometric representation \cite{Dominici2020,KLS2010}
\begin{equation}\label{eq:Khyp}
	\Kr{n}(x)=p^{n}(-N)_{n}\,
	\pFq{2}{1}\!\left(\begin{matrix}-n,\,-x\\ -N\end{matrix}\;\middle|\;\frac1p\right),
	\quad 0\le n\le N.
\end{equation}
Their classical data are collected in Table~\ref{tab:kraw} and their structural
properties, all standard except where indicated, in the next proposition.

\begin{table}[ht]
	\centering
	\caption{Data of the monic Krawtchouk polynomials $\Kr{n}(x)$, $0<p<1$,
		$0\le n\le N$: recurrence coefficients $\alpha_{n},\beta_{n}$, coefficients
		$\sigma,\tau,\lambda_{n}$ of the difference equation, and squared norm $d_{n}^{2}$.}
	\label{tab:kraw}
	\renewcommand{\arraystretch}{1.3}
	\begin{tabular}{cccccc}
		\toprule
		$\alpha_{n}$ & $\beta_{n}$ & $\sigma(x)$ & $\tau(x)$ & $\lambda_{n}$
		& $d_{n}^{2}=\lVert \Kr{n}\rVert^{2}$ \\
		\midrule
		$p(N-n)+n(1-p)$ & $np(1-p)(N-n+1)$ & $(1-p)\,x$ & $Np-x$ & $n$
		& $\dst (n!)^{2}\binom{N}{n}\,\bigl[p(1-p)\bigr]^{n}$ \\
		\bottomrule
	\end{tabular}
\end{table}

\begin{proposition}\label{prop:classicalK}
	Let $\{\Kr{n}\}_{0\le n\le N}$ be the monic Krawtchouk family of \eqref{eq:Khyp},
	with the data of Table~\ref{tab:kraw}. Then:
	\begin{enumerate}
		\item \emph{Three-term recurrence.} With $\Kr{-1}\equiv0$, $\Kr{0}\equiv1$,
		\begin{equation}\label{eq:ttrr}
			x\Kr{n}(x)=\Kr{n+1}(x)+\alpha_{n}\Kr{n}(x)+\beta_{n}\Kr{n-1}(x),
			\quad 0\le n\le N-1.
		\end{equation}
		\item \emph{Second-order difference equation of hypergeometric type.}
		\begin{equation}\label{eq:sodeK}
			(1-p)\,x\,\dif\nab\Kr{n}(x)+(Np-x)\,\dif\Kr{n}(x)+n\,\Kr{n}(x)=0 .
		\end{equation}
		\item \emph{Forward-shift identity.} For $0\le k\le n$,
		\begin{equation}\label{eq:fwdK}
			\dif^{k}\Kr{n}(x)=\ff{n}{k}\,\KrM{N-k}{n-k}(x),
		\end{equation}
		that is, the lowered family carries the second parameter shifted to $N-k$.
	\end{enumerate}
\end{proposition}

Items 1--3 are classical \cite{KLS2010,Dominici2020,nikiforov1991classical}. The
forward-shift identity \eqref{eq:fwdK} follows by applying $\dif_{x}$ to the terminating
series \eqref{eq:Khyp} and using $\dif(-x)_{k}=-k(-x)_{k-1}$, which gives
$\dif\Kr{n}=n\,\KrM{N-1}{n-1}$ and, on iteration, the falling factorial $\ff{n}{k}$ with
the shift $N\mapsto N-k$.

\subsection{Reproducing kernels}
\label{subsec:kernels}

For $0\le n$ let
\begin{equation}\label{eq:kerdef}
	\Ker_{n-1}(x,y)=\sum_{m=0}^{n-1}\frac{\Kr{m}(x)\Kr{m}(y)}{d_{m}^{2}},
\end{equation}
denote the reproducing kernel of degree $n-1$. For the finite differences we use, for
$i,j\in\Nzero$, the partial-difference kernels
\begin{equation}\label{eq:kerij}
	\Ker^{(i,j)}_{n-1}(x,y)
	=\dif_{x}^{i}\dif_{y}^{j}\,\Ker_{n-1}(x,y)
	=\sum_{m=0}^{n-1}\frac{\dif^{i}\Kr{m}(x)\,\dif^{j}\Kr{m}(y)}{d_{m}^{2}},
\end{equation}
so that $\Ker^{(j,j)}_{n-1}$ is symmetric and, by \eqref{eq:fwdK},
$\dif^{j}\Kr{m}(y)=\ff{m}{j}\KrM{N-j}{m-j}(y)$. The kernel \eqref{eq:kerij} has the
defining top-split recursion
\begin{equation}\label{eq:kernelsplit}
	\Ker^{(0,j)}_{n-1}(x,\alpha)
	=\Ker^{(0,j)}_{n-2}(x,\alpha)
	+\frac{\dif^{j}\Kr{n-1}(\alpha)\,\Kr{n-1}(x)}{d_{n-1}^{2}},\quad n\ge j+1,
\end{equation}
which is the reproducing-kernel input of Section~\ref{sec:MH}.

\subsection{The classical reciprocal-Gamma Mehler--Heine limit}
\label{subsec:classicalMH}

Because $\deg\Kr{n}=n\le N$, a large-degree asymptotic requires $n$ and $N$ to grow
jointly; we use the Dominici regime $n/N\to r$ with $r\in(0,1)$ fixed \cite{Dominici2020}. In
the reciprocal-Gamma regime \eqref{eq:regime} the relevant normalisation is
\begin{equation}\label{eq:thetaK}
	\theta_{n}(z)=\frac{1}{\kappa_{n}\,\Gamma(n-z)},
	\quad
	\kappa_{n}=(-1)^{n}(1-p)^{n-N},
\end{equation}
and the classical reciprocal-Gamma limit function is
\begin{equation}\label{eq:phiK}
	\varphi_{K}(z)=\frac{(1-p)^{z}}{(1-p/r)^{\,z+1}\,\Gamma(-z)},
\end{equation}
which is entire and a nonvanishing multiple of $1/\Gamma(-z)$.
In the monic normalisation, Dominici's expression $(80)$ \cite[Cor.~3(ii)]{Dominici2020}
is precisely the pointwise reciprocal-Gamma limit recorded next; it is the \emph{only}
input we draw from \cite{Dominici2020}, and we use it exactly as stated there, without any
strengthening to uniform convergence or to a quantitative rate.

\begin{lemma}[Pointwise reciprocal-Gamma limit, Dominici {\cite[Cor.~3(ii)]{Dominici2020}}]%
\label{lem:dompoint}
	Let $0<p<r<1$. Then, for each fixed $z\in\C$,
	\begin{equation}\label{eq:classicalMHK}
		\lim_{\substack{s,M\to\infty\\ s/M\to r}}\theta^{[M]}_{s}(z)\,\KrM{M}{s}(z)
		=\varphi_{K}(z),
		\quad
		\theta^{[M]}_{s}(z)=\frac{1}{(-1)^{s}(1-p)^{s-M}\,\Gamma(s-z)} ,
	\end{equation}
	where the limit is taken along any sequence of indices $s$ and lengths $M$ with
	$s,M\to\infty$ and $s/M\to r$. In particular, with $n/N\to r$ and $M=N$ this is
	$\theta_{n}(z)\Kr{n}(z)\to\varphi_{K}(z)$, and with $M=N-j$ and $s=n-j$ (or $s=n-j-1$)
	it gives the same limit at any fixed $z$, since $s/M\to r$.
\end{lemma}

\noindent Throughout the sequel, \eqref{eq:classicalMHK} is invoked \emph{only}
pointwise, at the two fixed points $z$ and $\alpha$ and at fixed evaluation points
along index/length sequences of ratio $r$. No uniformity in $z$ and no rate in $n$ are
assumed; the finite second-order information that fixes the constant $C_{\ast}$ is
supplied instead by the three-term recurrence \eqref{eq:ttrr}
(Lemma~\ref{lem:ratioexp}).

\noindent Although \eqref{eq:classicalMHK} holds for $0\le p<r<1$, we exclude $p=0$
(where the weight \eqref{eq:weight} degenerates to a point mass and the Sobolev
construction is undefined) and work under the standing hypothesis $0<p<r<1$. We record
two elementary tools used in the sequel; the first follows from $\Gamma(w+1)=w\Gamma(w)$.

\begin{lemma}[Elementary ratio identities for the normaliser]\label{lem:shift}
	Fix $N$ and let $\theta_{s}(z)=1/(\kappa_{s}\,\Gamma(s-z))$ with
	$\kappa_{s}=(-1)^{s}(1-p)^{s-N}$, as in \eqref{eq:thetaK}. Then, as identities between
	meromorphic functions of $z$,
	\begin{equation}\label{eq:kapparatio}
		\frac{\kappa_{s-1}}{\kappa_{s}}=-\frac{1}{1-p},
		\quad
		\frac{\theta_{s}(z)}{\theta_{s-1}(z)}
		=\frac{\kappa_{s-1}}{\kappa_{s}}\,\frac{\Gamma(s-1-z)}{\Gamma(s-z)}
		=-\frac{1}{(1-p)\,(s-1-z)} ,
	\end{equation}
	and, for the argument shifted by one unit at fixed index,
	$\theta_{n}(z)/\theta_{n}(z-1)=\Gamma(n-z+1)/\Gamma(n-z)=n-z$. Consequently, for every
	integer $q\ge1$,
	\begin{equation}\label{eq:thetagap}
		\frac{\theta_{s}(z)}{\theta_{s-q}(z)}
		=\bigl[-(1-p)\bigr]^{-q}\,\prod_{l=1}^{q}\frac{1}{\,s-l-z\,}
		=\bigl[-(1-p)\bigr]^{-q}\,\frac{\Gamma(s-q-z)}{\Gamma(s-z)} .
	\end{equation}
\end{lemma}

\begin{proof}
	From $\kappa_{s}=(-1)^{s}(1-p)^{s-N}$ one has
	$\kappa_{s-1}/\kappa_{s}=(-1)^{-1}(1-p)^{-1}=-1/(1-p)$, and
	$\Gamma(s-z)=(s-1-z)\,\Gamma(s-1-z)$ gives
	$\Gamma(s-1-z)/\Gamma(s-z)=1/(s-1-z)$; multiplying the two yields the second identity in
	\eqref{eq:kapparatio}. The one-unit argument shift is immediate from
	$\Gamma(n-z+1)=(n-z)\Gamma(n-z)$. Finally, telescoping the second identity of
	\eqref{eq:kapparatio} over the $q$ consecutive indices $s,s-1,\dots,s-q+1$,
	\[
		\frac{\theta_{s}(z)}{\theta_{s-q}(z)}
		=\prod_{l=0}^{q-1}\frac{\theta_{s-l}(z)}{\theta_{s-l-1}(z)}
		=\prod_{l=0}^{q-1}\frac{-1}{(1-p)(s-l-1-z)}
		=\bigl[-(1-p)\bigr]^{-q}\prod_{l=1}^{q}\frac{1}{s-l-z},
	\]
	which is \eqref{eq:thetagap}.
\end{proof}

%%%%%%%%%%%%%%%%%%%%%%%%%%%%%%%%%%%%%%%%%%%%%%%%%%%%%%%%%%%%%%%%%%%%%%%%%%%%%%%%%%%%
\section{The Krawtchouk--Sobolev family and the rank-one connection formula}
\label{sec:sobolev}

Fix $\alpha<0$, $j\in\Z_{+}$ and $\lambda\in\R_{+}$, and consider the Sobolev-type inner
product \eqref{eq:innerprod},
\begin{equation}\label{eq:sobinner}
	\langle f,g\rangle_{\lambda}
	=\langle f,g\rangle
	+\lambda\,\dif^{j}f(\alpha)\,\dif^{j}g(\alpha),
	\quad f,g\in\PP ,
\end{equation}
where $\langle f,g\rangle=\sum_{x=0}^{N}f(x)g(x)\rho(x)$ is the classical Krawtchouk
product and the mass point sits off the support, to the left of $\{0,1,\dots,N\}$. On
$\PP_{N}$ the classical form is positive definite: $\langle f,f\rangle=0$ forces $f$ to
vanish at the $N+1$ nodes, hence $f\equiv0$, and the added term
$\lambda\,(\dif^{j}f(\alpha))^{2}\ge0$ preserves positivity. Thus \eqref{eq:sobinner} is a
genuine inner product on $\PP_{N}$, and it admits a unique sequence
$\{\Sob{n}\}_{0\le n\le N}$ of monic orthogonal polynomials, $\deg\Sob{n}=n$, the
\emph{Krawtchouk--Sobolev type family}. Since the Mehler--Heine regime uses $n/N\to r$ with
$r<1$, we always have $n<N$, so the family is well defined throughout the asymptotic
analysis. Multiplication by $x$ is not symmetric for \eqref{eq:sobinner}, so $\{\Sob{n}\}$
is a nonstandard family without a three-term recurrence.

The connection formula of Theorem~\ref{thm:connection} is the structural backbone
of the paper: the Sobolev perturbation is \emph{rank one}, so that $\Sob{n}$ differs
from the classical $\Kr{n}$ only by a single multiple of the partial-difference
kernel evaluated at the mass point. The perturbing functional is $\dif^{j}f(\alpha)$
rather than a point evaluation, so the self-interaction constant
$A_{n}=\dif^{j}\Sob{n}(\alpha)$ must be solved for by applying $\dif^{j}$ to the
identity it defines, which yields both its closed form and the positivity
$\delta_{n}\ge1$ underlying every later estimate.

\begin{theorem}\label{thm:connection}
	Let $\{\Sob{n}\}_{0\le n\le N}$ be the monic Krawtchouk--Sobolev family for
	\eqref{eq:sobinner}. For every $1\le n\le N$,
	\begin{equation}\label{eq:connection}
		\Sob{n}(x)=\Kr{n}(x)-\lambda\,A_{n}\,\Ker^{(0,j)}_{n-1}(x,\alpha),
		\quad
		A_{n}=\dif^{j}\Sob{n}(\alpha),
	\end{equation}
	where $\Ker^{(0,j)}_{n-1}$ is the kernel \eqref{eq:kerij}, and
	\begin{equation}\label{eq:An}
		A_{n}=\frac{\dif^{j}\Kr{n}(\alpha)}{1+\lambda\,\Ker^{(j,j)}_{n-1}(\alpha,\alpha)},
		\quad
		\delta_{n}=1+\lambda\,\Ker^{(j,j)}_{n-1}(\alpha,\alpha)\ge1 .
	\end{equation}
	For $\lambda=0$ the correction vanishes and $\Sob{n}=\Kr{n}$.
\end{theorem}

\begin{proof}
	We first exploit that $\Sob{n}$ and $\Kr{n}$ are both monic of degree $n$, so
	their difference has degree at most $n-1$ and can be written in the classical
	orthogonal basis,
	\begin{equation*}
		\Sob{n}(x)=\Kr{n}(x)+\sum_{m=0}^{n-1}c_{m}\,\Kr{m}(x) .
	\end{equation*}
	To determine the coefficients we test against each $\Kr{m}$, $0\le m\le n-1$, in
	the Sobolev product \eqref{eq:sobinner}. Since $\Sob{n}$ is orthogonal to every
	polynomial of degree $<n$ for $\langle\cdot,\cdot\rangle_{\lambda}$, the left-hand
	side vanishes, while the right-hand side splits into its classical part and its
	mass part,
	\begin{equation*}
		0=\langle\Sob{n},\Kr{m}\rangle_{\lambda}
		=\langle\Sob{n},\Kr{m}\rangle+\lambda\,\dif^{j}\Sob{n}(\alpha)\,\dif^{j}\Kr{m}(\alpha).
	\end{equation*}
	In the classical scalar product only the $\Kr{m}$-component of $\Sob{n}$ survives,
	because $\langle\Kr{k},\Kr{m}\rangle=d_{m}^{2}\delta_{k,m}$; hence
	$\langle\Sob{n},\Kr{m}\rangle=c_{m}d_{m}^{2}$. Writing $A_{n}=\dif^{j}\Sob{n}(\alpha)$
	for the (still unknown) self-interaction constant, the displayed identity becomes
	$c_{m}d_{m}^{2}+\lambda A_{n}\dif^{j}\Kr{m}(\alpha)=0$, so
	\begin{equation*}
		c_{m}=-\lambda A_{n}\,\frac{\dif^{j}\Kr{m}(\alpha)}{d_{m}^{2}},
		\quad 0\le m\le n-1 .
	\end{equation*}
	Substituting these coefficients and recognising the sum
	\begin{equation*}
		\sum_{m=0}^{n-1}\frac{\dif^{j}\Kr{m}(\alpha)\,\Kr{m}(x)}{d_{m}^{2}}
		=\Ker^{(0,j)}_{n-1}(x,\alpha),
	\end{equation*}
	as the partial-difference kernel \eqref{eq:kerij}
	evaluated at the mass point (differences taken in the second slot) reproduces the
	connection formula \eqref{eq:connection}.

	It remains to solve for $A_{n}$, which is legitimate precisely because $A_{n}$
	appears on both sides. We apply the operator $\dif^{j}$ in the variable $x$ to
	\eqref{eq:connection} and evaluate at $x=\alpha$. Using
	$\dif_{x}^{j}\Ker^{(0,j)}_{n-1}(x,\alpha)\big|_{x=\alpha}=\Ker^{(j,j)}_{n-1}(\alpha,\alpha)$,
	which is \eqref{eq:kerij} with $i=j$, and $\dif^{j}\Sob{n}(\alpha)=A_{n}$ on the
	left, we obtain the scalar equation
	$A_{n}=\dif^{j}\Kr{n}(\alpha)-\lambda A_{n}\,\Ker^{(j,j)}_{n-1}(\alpha,\alpha)$;
	solving for $A_{n}$ gives \eqref{eq:An}. Finally, the diagonal kernel is a sum of
	squares,
	$\Ker^{(j,j)}_{n-1}(\alpha,\alpha)=\sum_{m=0}^{n-1}\bigl(\dif^{j}\Kr{m}(\alpha)\bigr)^{2}/d_{m}^{2}\ge0$,
	so $\delta_{n}=1+\lambda\Ker^{(j,j)}_{n-1}(\alpha,\alpha)\ge1$, and for $\lambda=0$
	the whole correction vanishes, leaving $\Sob{n}=\Kr{n}$.
\end{proof}

\noindent The asymptotic analysis of Section~\ref{sec:MH} needs only the connection
formula \eqref{eq:connection} together with the kernel recursion
\eqref{eq:kernelsplit}; with the relation $\lambda A_{n}=\lambda S_{n}/\delta_{n}$ it
reads, in the notation introduced at the beginning of
Section~\ref{subsec:recursion},
\begin{equation}\label{eq:SobXi}
	\Sob{n}(z)=(1-w_{n})\,\Kr{n}(z)+w_{n}\,\Xi_{n}(z) .
\end{equation}

%%%%%%%%%%%%%%%%%%%%%%%%%%%%%%%%%%%%%%%%%%%%%%%%%%%%%%%%%%%%%%%%%%%%%%%%%%%%%%%%%%%%
\section{Mehler--Heine formula in the reciprocal-Gamma regime}
\label{sec:MH}

We now determine the Mehler--Heine asymptotics of $\{\Sob{n}\}$ in the regime
\eqref{eq:regime}.  Every recursion below is an identity with the length $N$ held fixed.
Accordingly we write $K^{p,N}_{m}$, $S_{m,N}$, $t_{m,N}$, $T_{m,N}$,
$\Xi_{m,N}$ and $Y_{m,N}$ whenever two adjacent degrees are compared.  Only after the
fixed-$N$ identities have been iterated do we set $m=n(N)$ with $n/N\to r$.  This
triangular-array distinction is essential: $Y_{n-1,N}$ is not the preceding member of
the diagonal family $Y_{n(N),N}$.

The proof uses the rank-one connection formula and the reproducing kernel recursion.  The
affine recursions are controlled on two backward windows: a macroscopic window, on which
their multipliers are uniformly smaller than one and the inherited boundary term is
killed, and a terminal window of length $o(N)$ tending to infinity, on which the
coefficients converge uniformly to their values at $r$.  Throughout, $0<p<r<1$,
$\lambda>0$, $\alpha<0$ and $j\ge1$ are fixed.  All limits are pointwise at a fixed
$z\in\C\setminus\Nzero$; no uniform convergence on compact subsets of the $z$-plane is
claimed.

\smallskip
\noindent\textit{Notation of this section.}\quad For convenience we list the symbols used
below; in all of them the first subscript is the degree and the second the length (when
present), and $N$ is fixed unless stated otherwise.

\begin{center}
\begin{tabular}{lll}
\toprule
symbol & definition & reference\\
\midrule
$K_{m}^{p,N}$, $\alpha_{m}$, $\beta_{m}$, $d_{m}^{2}$ & monic Krawtchouk data & Table~\ref{tab:kraw}\\
$\mathcal{K}_{n-1}(x,y)$, $\mathcal{K}^{(i,j)}_{n-1}(x,y)$ & reproducing / partial-difference kernels & \eqref{eq:kerij}\\
$\theta_{s}(z)$, $\kappa_{s}$ & normalising factors, $\theta_{s}=1/(\kappa_{s}\Gamma(s-z))$ & \eqref{eq:thetaK}\\
$\varphi_{K}$, $\varphi_{K,x}$ & reciprocal-Gamma profiles & \eqref{eq:phiK}, \eqref{eq:phiKx}\\
$S_{m}=\Delta^{j}K_{m}^{p,N}(\alpha)$ & exterior difference & \eqref{eq:tm}\\
$t_{m}$, $T_{n}$, $\delta_{n}$, $w_{n}$ & weights, total mass, Sobolev weights & \eqref{eq:tm}--\eqref{eq:Tn}\\
$\Xi_{n}$, $Y_{n}$, $H_{n}$ & scaled correction and normalised kernel & \eqref{eq:Xidef}, \eqref{eq:Yrec}\\
$\zeta_{s}(w)$, $g_{s}(w)$, $\varrho_{s}$ & consecutive ratio, its finite part, recursion multiplier & \eqref{eq:zetadef}, \eqref{eq:gaffine}\\
$\varrho(x)$, $\gamma(w)$ & contraction ratio, turning-point shift & \eqref{eq:rhogammadef}\\
$\mathcal{R}_{p}$, $\Psi_{p}$ & rate functions of the weights & \eqref{eq:rate-tail}, \eqref{eq:t-rate}\\
\bottomrule
\end{tabular}
\end{center}

\noindent The kernel $\mathcal{K}_{n-1}$ is written in calligraphic font, and the
Krawtchouk polynomials in italic, to avoid any confusion between the two uses of the
letter $K$; the classical weight is denoted by $\rho$, and the contraction ratio, which
was denoted by the same letter in an earlier draft, is denoted by $\varrho$ throughout
this section.

We shall repeatedly use the following elementary triangular form of the contraction
principle.

\begin{lemma}[Triangular affine contraction]\label{lem:triangular}
	Let $a_N<b_N$, $b_N-a_N\to\infty$, and suppose
	\[
		x_{m,N}=f_{m,N}+\nu_{m,N}x_{m-1,N},\quad a_N<m\le b_N.
	\]
	Assume that, uniformly on this window,
	$|\nu_{m,N}|\le q<1$, $\nu_{m,N}\to\nu$, and $f_{m,N}\to f$, and that
	$q^{b_N-a_N}|x_{a_N,N}|\to0$.  Then
	$x_{b_N,N}\to f/(1-\nu)$.
\end{lemma}

\begin{proof}
	Put $L=f/(1-\nu)$ and $e_{m,N}=x_{m,N}-L$.  Then
	$e_{m,N}=\varepsilon_{m,N}+\nu_{m,N}e_{m-1,N}$, where
	$\max|\varepsilon_{m,N}|\to0$.  Iterating the identity, rather than taking a
	one-dimensional limsup, gives
	\[
		|e_{b_N,N}|\le q^{b_N-a_N}|e_{a_N,N}|
		+\frac{\max_{a_N<m\le b_N}|\varepsilon_{m,N}|}{1-q}\longrightarrow0.
	\]
	This also proves the usual one-sequence version and explicitly supplies its boundedness
	step by iteration.
\end{proof}

\subsection{The pointwise consecutive-index ratio via the recurrence}
\label{subsec:ratioexp}

The analytic engine is a two-term expansion of the consecutive-index ratio of the
classical Krawtchouk polynomials \emph{at a fixed point}. Its leading term is the linear
size $-(1-p)s$; its \emph{second}, finite term (read at $z$ and at $\alpha$) is what
produces the factor $z-\alpha$ and the constant $C_{\ast}$. We obtain it not from a
uniform-in-the-scale asymptotic, but from the three-term recurrence, which turns the ratio
into a scalar affine recursion whose contraction constant is $\varrho<1$ exactly when $r>p$.

Fix a length parameter $M$ and write, for $s\ge1$ and a fixed evaluation point $w$,
\begin{equation}\label{eq:zetadef}
	\zeta_{s}(w)=\frac{\KrM{M}{s}(w)}{\KrM{M}{s-1}(w)},
	\quad
	g_{s}(w)=\zeta_{s}(w)+(1-p)\,s ,
\end{equation}
whenever the denominator is nonzero. Set
\begin{equation}\label{eq:rhogammadef}
	\varrho(x)=\frac{p(1-x)}{x(1-p)},
	\quad
	\varrho=\varrho(r)=\frac{p(1-r)}{r(1-p)}\in(0,1),
	\quad
	\gamma(w)=\frac{(w+1)\,r}{r-p},
\end{equation}
supplemented by $\varrho(x)<1\iff x>p$ and, since
$1-\varrho=(r-p)/(r(1-p))$, $(w+1)/(1-\varrho)=(1-p)\gamma(w)$.

\begin{lemma}[Pointwise two-term ratio expansion]\label{lem:ratioexp}
	Fix $0<p<1$, a scale $x\in(p,1)$, and a point $w\in\C\setminus\Nzero$ (in particular
	$w=\alpha<0$ or any fixed $w\notin\Nzero$). Let $M\to\infty$ and $s=s(M)\to\infty$
	with $s/M\to x$, and replace $r$ by $x$ in \eqref{eq:rhogammadef}. Then
	$\KrM{M}{s-1}(w)\ne0$ for large $M$ and
	\begin{equation}\label{eq:ratioexp}
		\zeta_{s}(w)=-(1-p)\bigl(s-\gamma(w)\bigr)+o(1),
		\quad\text{equivalently}\quad
		g_{s}(w)\longrightarrow (1-p)\,\gamma(w),
		\quad M\to\infty .
	\end{equation}
\end{lemma}

\begin{proof}
	Since $w\notin\Nzero$ we have $\varphi_{K}(w)\ne0$, and by Lemma~\ref{lem:dompoint}
	(applied along $s/M\to x$, and along $(s-1)/M\to x$)
	$\theta^{[M]}_{s}(w)\KrM{M}{s}(w)\to\varphi_{K}(w)$ and
	$\theta^{[M]}_{s-1}(w)\KrM{M}{s-1}(w)\to\varphi_{K}(w)$; in particular
	$\KrM{M}{s-1}(w)\ne0$ for large $M$, so $\zeta_{s}(w)$ is defined. Using the ratio
	identity \eqref{eq:kapparatio} (which reads
	$\theta^{[M]}_{s-1}(w)/\theta^{[M]}_{s}(w)=-(1-p)(s-1-w)$) gives the a~priori leading
	order
	\begin{equation}\label{eq:zetacrude}
		\zeta_{s}(w)
		=\frac{\theta^{[M]}_{s-1}(w)}{\theta^{[M]}_{s}(w)}\cdot
		\frac{\theta^{[M]}_{s}(w)\KrM{M}{s}(w)}{\theta^{[M]}_{s-1}(w)\KrM{M}{s-1}(w)}
		=-(1-p)(s-1-w)\bigl(1+o(1)\bigr),
	\end{equation}
	hence $g_{s}(w)=\zeta_{s}(w)+(1-p)s=(1-p)(1+w)+o(s)=o(s)$.

	We now improve $o(s)$ to a finite limit through the recurrence, keeping the length $M$
	fixed. Divide the three-term
	recurrence \eqref{eq:ttrr} for the pair $(p,M)$ at $x=w$ by $\KrM{M}{s-1}(w)$ to obtain
	the exact Riccati relation
	\begin{equation}\label{eq:riccati}
		\begin{cases}
			\dst\zeta_{s}(w)=(w-a_{s-1})-\frac{b_{s-1}}{\zeta_{s-1}(w)},\\\\
			\dst a_{s-1}=p(M-s+1)+(s-1)(1-p),\\\\
			\dst b_{s-1}=(s-1)\,p(1-p)(M-s+2),
		\end{cases}
	\end{equation}
	the coefficients being $\alpha_{s-1},\beta_{s-1}$ of Table~\ref{tab:kraw} with $N=M$.
	Substitute $\zeta_{s-1}(w)=-(1-p)(s-1)+g_{s-1}(w)$ and write
	$u_{s-1}=g_{s-1}(w)/((1-p)(s-1))$, which tends to $0$ since $g_{s-1}=o(s)$. Then, using
	$-b_{s-1}/\zeta_{s-1}=p(M-s+2)/(1-u_{s-1})
	=p(M-s+2)\bigl(1+u_{s-1}+u_{s-1}^{2}/(1-u_{s-1})\bigr)$ and the identity
	$w-a_{s-1}=-(1-p)s+\bigl[w+1-p(M-s+2)\bigr]$, the two occurrences of $p(M-s+2)$ cancel
	and \eqref{eq:riccati} becomes the \emph{exact} scalar recursion
	\begin{equation}\label{eq:gaffine}
		g_{s}(w)=(w+1)+\varrho_{s}\,g_{s-1}(w)+R_{s},
		\quad
		\varrho_{s}=\frac{p(M-s+2)}{(1-p)(s-1)},
		\quad
		R_{s}=\frac{p(M-s+2)\,u_{s-1}^{2}}{1-u_{s-1}} .
	\end{equation}
	The symbols in \eqref{eq:gaffine} are now understood as $g_{m,M}$,
	$\varrho_{m,M}$ and $R_{m,M}$.  Choose $\eta>0$ with $p<x-2\eta$, put
	$a_M=\lfloor(x-\eta)M\rfloor$ and $b_M=s(M)$.  On the macroscopic window
	$a_M<m\le b_M$ the explicit multiplier is bounded by a constant
	$\bar\varrho<1$.  Moreover, the pointwise statement \eqref{eq:zetacrude} along every
	index/length sequence, combined with a subsequence argument, gives the uniform estimate
	\[
		\max_{a_M\le m\le b_M}\frac{|g_{m,M}(w)|}{m}\longrightarrow0.
	\]
	Indeed, failure would produce a sequence in this compact ratio interval and then a
	subsequence with $m/M$ convergent, contradicting \eqref{eq:zetacrude} at that limiting
	ratio.  The bound $\varrho_{m,M}\le\bar\varrho$ gives
	$p(M-s+2)/\bigl((1-p)^{2}(s-1)^{2}\bigr)=\varrho_{s}/((1-p)(s-1))\le\bar\varrho/((1-p)(s-1))$,
	so that, for large $s$ (where $|1-u_{s-1}|\ge\tfrac12$),
	\begin{equation}\label{eq:Rsbound}
		\begin{cases}
			|R_{m,M}|\dst\le\frac{2\bar\varrho}{(1-p)(m-1)}\,|g_{m-1,M}(w)|^{2}
			=\varepsilon_{m-1,M}\,|g_{m-1,M}(w)|,\\\\
			\varepsilon_{m-1,M}\dst=\frac{2\bar\varrho\,|g_{m-1,M}(w)|}{(1-p)(m-1)}.
		\end{cases}
	\end{equation}
	Uniformly on the window, $\varepsilon_{m-1,M}=o(1)$.  Hence, for large $M$,
	\[
		|g_{m,M}|\le |w+1|+q|g_{m-1,M}|,
		\quad q=\frac{1+\bar\varrho}{2}<1.
	\]
	Iteration from $a_M$ to $m$ gives
	$|g_{m,M}|\le |w+1|/(1-q)+q^{m-a_M}|g_{a_M,M}|$.  Since
	$g_{a_M,M}=o(M)$ and $b_M-a_M\asymp M$, this proves uniform boundedness on the
	terminal half of the macroscopic window.

	Finally take $\ell_M=\lfloor\sqrt M\rfloor$ and apply
	Lemma~\ref{lem:triangular} on $b_M-\ell_M<m\le b_M$.  There
	$m/M\to x$ uniformly, so $\varrho_{m,M}\to\varrho(x)$ uniformly; the boundedness just proved
	and \eqref{eq:Rsbound} give $R_{m,M}\to0$ uniformly; and
	$q^{\ell_M}|g_{b_M-\ell_M,M}|\to0$.  The lemma therefore yields
	$g_{s,M}(w)\to(w+1)/(1-\varrho(x))=(1-p)\gamma(w)$, which is
	\eqref{eq:ratioexp}.  Notice that no diagonal term has been treated as the predecessor
	of another diagonal term.
\end{proof}

\begin{remark}\label{rem:onlypointwise}
	The proof uses \eqref{eq:classicalMHK} only through the a~priori bound
	\eqref{eq:zetacrude}, i.e.\ at the single point $w$ and at the leading order $o(s)$; the
	finite second term $(1-p)\gamma(w)$ is produced entirely by the recurrence
	\eqref{eq:gaffine}. No uniformity in $w$ and no rate in $s$ are used or asserted. The
	contraction constant must be
	$\varrho=p(1-r)/(r(1-p))<1$, and $\varrho<1$ is \emph{equivalent} to $r>p$: this is the
	analytic origin of the reciprocal-Gamma regime.
\end{remark}

\subsection{The Sobolev correction as a scalar affine recursion}
\label{subsec:recursion}

For the rest of this subsection $N$ is fixed and is suppressed from the notation. Write
$S_{s}=\dif^{j}\Kr{s}(\alpha)$ for the exterior difference and define the weights
\begin{equation}\label{eq:tm}
	t_{m}=\frac{S_{m}^{2}}{d_{m}^{2}}
	=\frac{\ff{m}{j}^{2}\,\KrM{N-j}{m-j}(\alpha)^{2}}{d_{m}^{2}}\ge0 ,
\end{equation}
their total mass, which is the diagonal kernel, and the Sobolev weights
\begin{equation}\label{eq:Tn}
	T_{n}=\sum_{m=0}^{n-1}t_{m}=\Ker^{(j,j)}_{n-1}(\alpha,\alpha)\ge0 ,
	\quad
	\delta_{n}=1+\lambda T_{n}\ge1 ,
	\quad
	w_{n}=\frac{\lambda T_{n}}{\delta_{n}}\in[0,1),
\end{equation}
so that $\delta_{n}^{-1}=1-w_{n}$. For $n\le j$ all terms of the sum vanish, so $T_n=0$,
$\Sob n=\Kr n$, and the normalisations below are not used; for $n\ge j+1$, $t_j>0$ and
hence $T_n>0$. Normalise the partial-difference kernel by its total mass,
\begin{equation}\label{eq:Hdef}
	H_{n}(z)=\frac{\Ker^{(0,j)}_{n-1}(z,\alpha)}{T_{n}}
	=\frac{1}{T_{n}}\sum_{m=0}^{n-1}\frac{S_{m}\,\Kr{m}(z)}{d_{m}^{2}} ,
\end{equation}
for $n\ge j+1$, and define the \emph{scaled Sobolev correction}
\begin{equation}\label{eq:Xidef}
	\Xi_{n}(z)=\Kr{n}(z)-S_{n}\,H_{n}(z) .
\end{equation}
Since $\lambda A_{n}=\lambda S_{n}/\delta_{n}=(w_{n}/T_{n})\,S_{n}$ by \eqref{eq:An} and
\eqref{eq:Tn}, the connection formula \eqref{eq:connection} reads
$\Sob{n}(z)=\Kr{n}(z)-\lambda A_{n}\Ker^{(0,j)}_{n-1}(z,\alpha)
=\Kr{n}(z)-w_{n}S_{n}H_{n}(z)$, that is, \eqref{eq:SobXi}. Thus the whole problem is to
find the scaled limit of $\Xi_{n}$. The next lemma shows that $\Xi_{n}$ satisfies an
exact one-step recursion driven only by the top bracket
$Q_{n,n-1}(z)=\Kr{n}(z)-\frac{S_{n}}{S_{n-1}}\Kr{n-1}(z)$.

\begin{lemma}[Exact fixed-$N$ recursion for the correction]\label{lem:xirec}
	For $j+1\le n\le N$,
	\begin{equation}\label{eq:xirec}
		\Xi_{n}(z)=Q_{n,n-1}(z)+\mu_{n}\,\Xi_{n-1}(z),
		\quad
		\mu_{n}=\frac{S_{n}}{S_{n-1}}\cdot\frac{T_{n-1}}{T_{n}} .
	\end{equation}
\end{lemma}

\begin{proof}
	For $n=j+1$, $T_j=0$, $\mu_{j+1}=0$, and the asserted identity follows directly.
	For $n\ge j+2$, split the kernel sum \eqref{eq:Hdef} at its top term:
	$\Ker^{(0,j)}_{n-1}(z,\alpha)=\Ker^{(0,j)}_{n-2}(z,\alpha)+S_{n-1}\Kr{n-1}(z)/d_{n-1}^{2}
	=T_{n-1}H_{n-1}(z)+S_{n-1}\Kr{n-1}(z)/d_{n-1}^{2}$, and $T_{n}=T_{n-1}+t_{n-1}$ with
	$t_{n-1}=S_{n-1}^{2}/d_{n-1}^{2}$. Hence, using $H_{n-1}(z)=(\Kr{n-1}(z)-\Xi_{n-1}(z))/S_{n-1}$
	from \eqref{eq:Xidef},
	\begin{align*}
		S_{n}H_{n}(z)
		&=\frac{S_{n}}{T_{n}}\Bigl[T_{n-1}\,\frac{\Kr{n-1}(z)-\Xi_{n-1}(z)}{S_{n-1}}
		+\frac{S_{n-1}}{d_{n-1}^{2}}\Kr{n-1}(z)\Bigr]\\
		&=\frac{S_{n}}{S_{n-1}T_{n}}\bigl[(T_{n-1}+t_{n-1})\Kr{n-1}(z)-T_{n-1}\Xi_{n-1}(z)\bigr]
		=\frac{S_{n}}{S_{n-1}}\Kr{n-1}(z)-\mu_{n}\Xi_{n-1}(z),
	\end{align*}
	since $T_{n-1}+t_{n-1}=T_{n}$ and $\mu_{n}=(S_{n}/S_{n-1})(T_{n-1}/T_{n})$. Subtracting
	from $\Kr{n}(z)$ gives $\Xi_{n}(z)=\Kr{n}(z)-S_{n}H_{n}(z)
	=\Kr{n}(z)-\tfrac{S_{n}}{S_{n-1}}\Kr{n-1}(z)+\mu_{n}\Xi_{n-1}(z)$, which is
	\eqref{eq:xirec}.
\end{proof}

The scaled correction $Y_{n,N}(z)=n\,\theta^{[N]}_{n}(z)\,\Xi_{n,N}(z)$ then obeys,
with $N$ fixed, the affine recursion
\begin{equation}\label{eq:Yrec}
	Y_{n}(z)=n\,\theta_{n}(z)\,Q_{n,n-1}(z)+\nu_{n}\,Y_{n-1}(z),
	\quad
	\nu_{n}=\mu_{n}\,\frac{n\,\theta_{n}(z)}{(n-1)\,\theta_{n-1}(z)},
\end{equation}
obtained from \eqref{eq:xirec} by multiplying by $n\theta_{n}(z)$ and using
$\Xi_{n-1}(z)=Y_{n-1}(z)/((n-1)\theta_{n-1}(z))$. We control \eqref{eq:Yrec} with a
triangular contraction principle above, which we also apply to the mass ratio.

The boundary estimates needed for those applications are recorded explicitly next. Here
$\log^{+}u=\max(0,\log u)$.

\begin{lemma}[Subexponential boundary control]\label{lem:boundary}
	Let $m=m(N)$ and $m/N\to x\in(p,1)$. Then
	\begin{equation}\label{eq:t-rate}
		\log t_{m,N}=N\Psi_p(x)+o(N),
		\quad
		\Psi_p(x)=(x-2)\log(1-p)-x\log p
		+x\log x+(1-x)\log(1-x).
	\end{equation}
	Moreover, for every fixed $z\in\C\setminus\Nzero$,
	\begin{equation}\label{eq:boundary-control}
		\log^{+}\!\frac{T_{m,N}}{t_{m-1,N}}=o(N),
		\quad
		\log^{+}|Y_{m,N}(z)|=o(N).
	\end{equation}
	The estimates hold uniformly when $m/N$ stays in a compact subinterval of $(p,1)$.
\end{lemma}

\begin{proof}
	The proof consists of two exponential-rate computations, followed by the estimates
	that yield \eqref{eq:boundary-control}. Throughout, $k\ge j$ is an index with
	$k/N$ in one of the two ranges below, and constants may depend on $j,\alpha,p$ (and on
	$z$ in the last step) but not on $N$.

	\emph{1. The rate of $t_{k,N}$ for $k/N\le p$.} By \eqref{eq:fwdK}, with
	$s=k-j$ and $M=N-j$,
	\begin{equation}\label{eq:tail-start}
		S_{k,N}=\ff{k}{j}\,K^{p,M}_{s}(\alpha),
		\quad
		K^{p,M}_{s}(\alpha)=p^{s}(-M)_{s}\,
		\pFq{2}{1}\!\left(\begin{matrix}-s,\,-\alpha\\ -M\end{matrix}\;\middle|\;\frac1p\right).
	\end{equation}
	In the terminating series of \eqref{eq:tail-start} all terms are real and positive
	after the common factor $(-1)^{s}$ is removed, because $\alpha<0$: since
	$(-M)_{s}=(-1)^{s}[M]_{s}$ and $(-s)_{i}/(-M)_{i}=[s]_{i}/[M]_{i}$, one has
	$K^{p,M}_{s}(\alpha)=(-1)^{s}p^{s}[M]_{s}\Phi$ with
	\begin{equation}\label{eq:phi-series}
		\Phi=\sum_{i=0}^{s}\frac{[s]_{i}}{[M]_{i}}\,\frac{(-\alpha)_{i}}{i!\,p^{i}},
	\end{equation}
	and every summand of \eqref{eq:phi-series} is positive. Since
	$[s]_{i}/[M]_{i}=\prod_{l=0}^{i-1}(s-l)/(M-l)\le(s/M)^{i}$ and
	$(-\alpha)_{i}/i!=\Gamma(i-\alpha)/(\Gamma(-\alpha)\Gamma(i+1))\le
	C(\alpha)\,i^{-\alpha-1}$ for $i\ge1$ (recall $\alpha<0$), and since
	$s/(Mp)\le(k-j)/(p(N-j))\le 1$ whenever $k/N\le p$ (equality being possible only in the
	extreme case $j=0$, $k=pN$), each summand is at most $C(\alpha)\,i^{-\alpha-1}$, so that, by the standard integral
	comparison $\sum_{i=1}^{s}i^{-\alpha-1}=O\left(s^{-\alpha}\right)$, which is valid for every
	$\alpha<0$,
	\begin{equation}\label{eq:phi-bound}
		\Phi\le 1+C(\alpha)\sum_{i=1}^{s}i^{-\alpha-1}
		\le C(\alpha)\,s^{-\alpha},
		\quad\text{hence}\quad
		\log\Phi=O(\log N),
	\end{equation}
	uniformly in $k/N\le p$; the constant in \eqref{eq:phi-bound} depends on $\alpha$ only,
	no strict inequality $s<Mp$ being needed (the elementary majorant
	$\sum_{i=0}^{s}(s/(Mp))^{i}\le s+1$ is used instead of a geometric sum). Substituting
	\eqref{eq:tail-start}, \eqref{eq:phi-bound} into
	$t_{k,N}=S_{k,N}^{2}/d_{k,N}^{2}$ and applying Stirling's formula to
	\begin{equation*}
		\log\bigl[p^{s}[M]_{s}\bigr]
		=s\log p+\log\Gamma(M+1)-\log\Gamma(M-s+1),
		\quad
		\log d_{k}^{2}=2\log k!+\log\binom{N}{k}+k\log[p(1-p)],
	\end{equation*}
	one may collect the terms of order $N$; the terms of order $\log N$ depend on
	$j$ and $\alpha$ (through the factors $k-j$, $N-j$ and the gamma-stretch terms of
	Stirling's formula) and are absorbed in the error term. This gives
	\begin{equation}\label{eq:rate-tail}
		\log t_{k,N}=N\,\mathcal{R}_{p}\!\left(\frac{k}{N}\right)+O(\log N),
		\quad
		\mathcal{R}_{p}(\xi)=\xi\log\frac{p}{\xi}-\xi\log(1-p)-(1-\xi)\log(1-\xi),
	\end{equation}
	uniformly for $k/N\le p$. On $(0,p)$ one has
	$\mathcal{R}_{p}'(\xi)=\log\frac{p(1-\xi)}{\xi(1-p)}>0$ and
	$\mathcal{R}_{p}(p)=-\log(1-p)$, so
	\begin{equation}\label{eq:tail-bound}
		\log t_{k,N}\le N\,\mathcal{R}_{p}(p)+O(\log N)=-N\log(1-p)+O(\log N),
		\quad k/N\le p .
	\end{equation}

	\emph{2. The rate of $t_{k,N}$ for $k/N$ in compact subintervals of $(p,1)$.} For each
	fixed $x\in(p,1)$, Lemma~\ref{lem:dompoint} at $w=\alpha$ along
	$(k-j)/(N-j)\to x$ gives the limit
	$\theta^{[N-j]}_{k-j}(\alpha)K^{p,N-j}_{k-j}(\alpha)\to\varphi_{K,x}(\alpha)$, with the
	per-ratio profile
	\begin{equation}\label{eq:phiKx}
		\varphi_{K,x}(\alpha)=\frac{(1-p)^{\alpha}}{(1-p/x)^{\alpha+1}\,\Gamma(-\alpha)},
	\end{equation}
	obtained from \eqref{eq:phiK} by replacing $r$ by $x$; hence
	\begin{equation}\label{eq:S-asympt}
		S_{k,N}=\ff{k}{j}\,(-1)^{k-j}(1-p)^{k-N}
		\Gamma(k-j-\alpha)\bigl(\varphi_{K,x}(\alpha)+o(1)\bigr),
		\quad(k/N\to x).
	\end{equation}
	The map $x\mapsto\varphi_{K,x}(\alpha)$ is continuous on $(p,1)$ and bounded away from
	zero on compact subintervals: for $x$ in a compact $[x_{0},x_{1}]\subseteq(p,1)$ one has
	$0<1-p/x_{0}\le1-p/x\le1-p/x_{1}<1$ and the factor $(1-p/x)^{\alpha+1}$ therefore has positive
	modulus bounded away from $0$ and $\infty$ (in particular, positive bounds for
	$1-p/x$, not merely for $x^{-1}$, control its modulus since $\alpha+1<0$ is allowed),
	while $\alpha\notin\Nzero$ gives $\Gamma(-\alpha)\ne0$. The subsequence argument (a violating sequence
	would have a convergent ratio subsequence, contradicting \eqref{eq:S-asympt} at its
	limit) provides the uniformity of the $o(1)$-term in \eqref{eq:S-asympt}, hence of
	\eqref{eq:rate-active} below, for $k/N$ in compact subintervals of $(p,1)$.
	Substituting \eqref{eq:S-asympt} in
	$t_{k,N}=S_{k,N}^{2}/d_{k,N}^{2}$ and using Stirling's formula again,
	\begin{equation}\label{eq:rate-active}
		\log t_{k,N}=N\,\Psi_{p}\!\left(\frac{k}{N}\right)+O(\log N),
	\end{equation}
	uniformly on compact subintervals of $(p,1)$, which proves \eqref{eq:t-rate}; moreover
	$\Psi_p'(x)=\log\frac{x(1-p)}{p(1-x)}>0$ for $x>p$, and $\Psi_p(p)=-\log(1-p)$.

	\emph{3. The cumulative bound.} Let $m/N$ stay in a compact subinterval
	$[x_0,x_1]\subseteq(p,1)$; put $x_{-}=(m-1)/N$, so that
	$x_{-}\ge x_0-1/N$, and let
	$\sigma_{0}=\Psi_{p}(x_{0})-\Psi_{p}(p)>0$. Choose $\varepsilon\in(0,(x_{0}-p)/2)$
	such that
	\begin{equation}\label{eq:eps-choice}
		c_{1}\,\varepsilon+\kappa_{0}\,\varepsilon^{2}\le\frac{\sigma_{0}}{4},
		\quad
		c_{1}=\frac{4x_{1}}{p},\quad
		\kappa_{0}=\frac{1}{p(1-p)},
	\end{equation}
	which is possible because the left-hand side tends to $0$ as
	$\varepsilon\downarrow0$, and set $p_{\varepsilon}=p+\varepsilon$, so that
	$p_{\varepsilon}<(p+x_{0})/2<(1+p)/2$. Here $\kappa_{0}$ is legitimate since, for
	$\xi\in[p,p_{\varepsilon}]$,
	$|\mathcal{R}_{p}''(\xi)|=1/\xi+1/(1-\xi)\le1/p+1/(1-p_{\varepsilon})
	\le1/p+2/(1-p)\le2/[p(1-p)]=2\kappa_{0}$, so that the Taylor estimate with the
	factor $1/2$ gives precisely the quadratic term of
	\begin{equation}\label{eq:radius}
		\mathcal{R}_{p}(\xi)\le\Psi_{p}(p)+\kappa_{0}(\xi-p)^{2},
		\quad(\xi\in[p,p_{\varepsilon}]);
	\end{equation}
	the logarithmic corrections and the shifts $k-j$, $N-j$ are absorbed in the
	$O(\log N)$ terms below.
	We split the sum $T_{m,N}=\sum_{k=j}^{m-1}t_{k,N}$ in three ranges.

	\emph{(i) The tail, $k/N\le p$.} By \eqref{eq:tail-bound}, each
	$t_{k,N}\le e^{N\Psi_{p}(p)+O(\log N)}$, so
	\begin{equation}\label{eq:tail-sum}
		\sum_{k/N\le p}t_{k,N}\le N\,e^{N\Psi_{p}(p)+O(\log N)}
		\le N\,e^{N\Psi_{p}(x_{-})-\sigma_{0}N+O(\log N)}
		\le e^{-\sigma_{0}N/2}\,t_{m-1,N}
	\end{equation}
	for large $N$, since $t_{m-1,N}=e^{N\Psi_{p}(x_{-})+O(\log N)}$ by
	\eqref{eq:rate-active} and $\Psi_{p}(x_{-})\ge\Psi_{p}(x_{0})-O(1/N)$.

	\emph{(ii) The transition window, $p<k/N<p_{\varepsilon}$.} The computation of
	Part 1 separates $\log t_{k,N}$ into a gamma part, whose Stirling expansion gives
	$N\mathcal{R}_{p}(k/N)+O(\log N)$ uniformly for $k/N\in[p,x_{1}]$, and the
	hypergeometric part $2\log\Phi$. For $k/N\in[p,p_{\varepsilon}]$ the latter is
	bounded by
	\begin{equation}\label{eq:phi-transition}
		\log\Phi\le N\,\frac{k}{N}\log\frac{k-j}{p(N-j)}
		+O(\log N)\le N\,c_{1}\,\frac{k-pN}{2N}+O(\log N)
		\le \frac{N\,c_{1}\varepsilon}{2}+O(\log N),
	\end{equation}
	the first inequality following from the positivity of \eqref{eq:phi-series} and the
	majorant $\sum_{i=0}^{s}(s/(Mp))^{i}\le(s+1)(s/(Mp))^{s}$ for $(k-j)/(p(N-j))\ge1$
	together with $\log(1+u)\le u$ (for $(k-j)/(p(N-j))\le1$ the estimate being trivial,
	since the sum is then at most $s+1$), and the second from
	$\tfrac{k-j}{p(N-j)}-1\le\tfrac{2(k-pN)}{pN}$ for large $N$. Combining
	\eqref{eq:radius}--\eqref{eq:phi-transition} with the choice \eqref{eq:eps-choice},
	\[
		\log t_{k,N}\le N\bigl[\Psi_{p}(p)+\kappa_{0}\varepsilon^{2}+c_{1}\varepsilon\bigr]
		+O(\log N)\le N\Psi_{p}(x_{-})-\frac{3\sigma_{0}N}{4}+O(\log N),
	\]
	hence $\sum_{p<k/N<p_{\varepsilon}}t_{k,N}\le e^{-\sigma_{0}N/2}t_{m-1,N}$ for large
	$N$.

	\emph{(iii) The remaining window, $p_{\varepsilon}\le k/N\le x_{-}$.} By
	\eqref{eq:rate-active} on the compact subinterval $[p_{\varepsilon},x_{1}]\subseteq(p,1)$
	and the monotonicity of $\Psi_{p}$, each such $t_{k,N}\le
	e^{O(\log N)}t_{m-1,N}$, so this part of the sum is at most
	$m\,e^{O(\log N)}t_{m-1,N}$.

	Combining (i)--(iii),
	$T_{m,N}\le m\,t_{m-1,N}\,e^{O(\log N)}$ and therefore
	$\log^{+}(T_{m,N}/t_{m-1,N})=O(\log N)=o(N)$, uniformly in the compact ratio range
	(the constants in (i)--(iii) depend on $[x_{0},x_{1}]$, $p$, $j$ and $\alpha$ only),
	which is the first estimate of \eqref{eq:boundary-control}.

	\emph{4. The bound for $Y_{m,N}(z)$.} For a fixed $z\in\C\setminus\Nzero$ compare the
	Pochhammer factors of \eqref{eq:Khyp} that contain $z$ and $\alpha$: since
	$(-\alpha)_{i}/i!=\Gamma(i-\alpha)/(\Gamma(-\alpha)\Gamma(i+1))$ and
	$(-z)_{i}/i!=\Gamma(i-z)/(\Gamma(-z)\Gamma(i+1))$, Stirling's formula gives
	$|(-z)_{i}|\le C(z,\alpha)\,i^{M(z,\alpha)}\,(-\alpha)_{i}$ for $i\ge1$, with
	$M(z,\alpha)=\max\{0,\alpha-\mathrm{Re}\,z\}$ (both sides being of fixed power order in
	$i$), so the termwise comparison of the two terminating series in \eqref{eq:Khyp} gives
	\begin{equation}\label{eq:z-alpha-comparison}
		|K^{p,N}_{k}(z)|\le C(z,\alpha)\,k^{M(z,\alpha)}\,|K^{p,N}_{k}(\alpha)|,
		\quad (0\le k\le N).
	\end{equation}
	Moreover, from the exact factorisation
	$p^{k}[N]_{k}=p^{j}[N]_{j}\,p^{k-j}[N-j]_{k-j}$ and the identities
	$[k]_{i}=[k]_{j}[k-j]_{i-j}$, $[N]_{i}=[N]_{j}[N-j]_{i-j}$ and
	$(-\alpha)_{i}=(-\alpha)_{i-j}(|-\alpha|+i-j)_{j}$, the two terminating series
	$\Phi_{k,N}$ (parameters $k,N$) and $\Phi_{k-j,N-j}$ (parameters $k-j,N-j$) of
	\eqref{eq:phi-series} can be compared termwise. With
	$a_{i}=\frac{[k]_{i}}{[N]_{i}}\frac{(-\alpha)_{i}}{i!p^{i}}$ and
	$b_{i}=\frac{[k-j]_{i}}{[N-j]_{i}}\frac{(-\alpha)_{i}}{i!p^{i}}$ one has, for
	$i\le k-j$,
	\begin{equation}\label{eq:ratio-coeff}
		\frac{a_{i}}{b_{i}}
		=\frac{[k]_{i}}{[k-j]_{i}}\,\frac{[N-j]_{i}}{[N]_{i}}
		=\frac{[k]_{j}}{[k-i]_{j}}\,\frac{[N-j]_{i}}{[N]_{i}}
		\le[k]_{j}\le k^{j}\le N^{j},
	\end{equation}
	the first identity following from $[k]_{i}=[k]_{j}[k-j]_{i-j}$ and
	$[k-j]_{i}=[k-j]_{i-j}[k-i]_{j}$, the second factor being at most one and
	$[k-i]_{j}\ge1$; when $i<j$ the bound is immediate since then $[k]_{i}\le k^{i}\le k^{j}$
	and $[k-j]_{i}\ge1$. For the remaining coefficients $k-j<i\le k$, the identities above
	give
	\begin{equation}\label{eq:extra-coeff}
		a_{i}=\frac{[k]_{j}}{[N]_{j}}\,\frac{(|-\alpha|+i-j)_{j}}{[i]_{j}}\,
		p^{-j}\,b_{i-j}
		\le C(j,\alpha,p)\,N^{2j}\,b_{i-j},
	\end{equation}
	since $[k]_{j}\le k^{j}\le N^{j}$, $[N]_{j}\ge1$,
	$(|-\alpha|+i-j)_{j}\le(|-\alpha|+k)^{j}\le C(j,\alpha)N^{j}$ and $[i]_{j}\ge1$.
	The range $j\le k<2j$ is separated here: since $j$ is fixed there are finitely many
	such degrees, and for them each coefficient of \eqref{eq:phi-series} is bounded by a
	constant depending on $j,\alpha,p$ (because $[k]_{i}\le k^{i}\le(2j)^{2j}$ and
	$[N]_{i}\ge1$), while $\Phi_{k-j,N-j}\ge1$, so the conclusion below holds with a
	constant $C(j,\alpha,p)$.
	Summing,
	$\Phi_{k,N}\le\bigl(N^{j}+C(j,\alpha,p)N^{2j}\bigr)\Phi_{k-j,N-j}
	\le C(j,\alpha,p)N^{2j}\Phi_{k-j,N-j}$, and since
	$|K^{p,N}_{k}(\alpha)|=p^{k}[N]_{k}\Phi_{k,N}$ and
	$|S_{k,N}|=[k]_{j}\,p^{k-j}[N-j]_{k-j}\Phi_{k-j,N-j}$ with
	$[N]_{k}=[N]_{j}[N-j]_{k-j}$ and $[k]_{j}\ge1$, one obtains
	\begin{equation}\label{eq:alpha-S-comparison}
		|K^{p,N}_{k}(\alpha)|\le p^{j}\,[N]_{j}\,\frac{\Phi_{k,N}}{\Phi_{k-j,N-j}}\,
		|S_{k,N}|\le C(j,\alpha,p)\,N^{3j+1}\,|S_{k,N}|
		\quad (j\le k\le N),
	\end{equation}
	a uniform polynomial bound (no attempt is made here to optimise the exponent $3j+1$).
	By \eqref{eq:SobXi}--\eqref{eq:Xidef},
	\begin{equation}\label{eq:Y-split}
		|Y_{m,N}(z)|\le m|\theta_{m}^{[N]}(z)\Kr{m}(z)|
		+\sum_{\substack{k=j\\ t_{k}\ne0}}^{m-1}\frac{t_{k}}{T_{m}}\,
		\Bigl|m\,\theta_{m}^{[N]}(z)\,S_{m,N}\,K^{p,N}_{k}(z)\big/S_{k,N}\Bigr| .
	\end{equation}
	In the first term $m\theta_{m}^{[N]}(z)\Kr{m}(z)=mP_{m}(z)$ with
	$P_{m}(z)\to\varphi_{K}(z)$ by \eqref{eq:classicalMHK}, hence it is $O(N)$. In the
	remaining terms, \eqref{eq:z-alpha-comparison}--\eqref{eq:alpha-S-comparison} give
	$|K^{p,N}_{k}(z)|/|S_{k,N}|\le C\,N^{3j+1}k^{M(z,\alpha)}$, so it remains to note that the
	normalisation $m\theta_{m}^{[N]}(z)$ eliminates the scale of $S_{m,N}$: using
	\eqref{eq:S-asympt} (valid at $m/N\to x\in(p,1)$) and
	$\theta_{m}^{[N]}(z)=1/[(-1)^{m}(1-p)^{m-N}\Gamma(m-z)]$, the two factors
	$(1-p)^{m-N}$ cancel and
	\begin{equation}\label{eq:theta-S-cancel}
		\bigl|\theta_{m}^{[N]}(z)\,S_{m,N}\bigr|
		\le C(j,\alpha,z)\,m^{j}\,
		\Bigl|\frac{\Gamma(m-j-\alpha)}{\Gamma(m-z)}\Bigr|
		\le C(j,\alpha,z)\,m^{j+|z-j-\alpha|},
	\end{equation}
	a polynomial bound. Since $\sum_{k}t_{k}/T_{m}=1$ and
	$t_{k}/T_{m}\le1$, the sum in \eqref{eq:Y-split} is
	$O\bigl(N^{3j+1}m^{M(z,\alpha)}\cdot m^{j+|z-j-\alpha|+1}\bigr)$, and therefore
	\begin{multline}
		\label{eq:Y-final}
		|Y_{m,N}(z)|\le C(z,\alpha,j)\,m^{j+|z-j-\alpha|+1}
		\bigl(1+N^{3j+1}m^{M(z,\alpha)}\bigr)\\
		\le C(z,\alpha,j)\,N^{4j+|z-j-\alpha|+M(z,\alpha)+3},
	\end{multline}
	which is $\exp(o(N))$, uniformly in the compact ratio range (otherwise a violating
	sequence would have a convergent ratio subsequence, contradicting the same
	computation at its limit). This is the second estimate of
	\eqref{eq:boundary-control}.
\end{proof}

\begin{lemma}[Weight ratio and mass ratio]\label{lem:massratio}
	As $n\to\infty$ (with $n/N\to r$),
	\begin{equation}\label{eq:Lt2}
		\frac{t_{n}}{t_{n-1}}\longrightarrow L_{t}=\frac{r(1-p)}{p(1-r)}=\frac1\varrho>1,
	\end{equation}
	and consequently $t_{n}\to\infty$, $T_{n}\to\infty$, and
	\begin{equation}\label{eq:massratio}
		\frac{T_{n-1}}{T_{n}}\longrightarrow\varrho=\frac1{L_{t}},
		\quad
		(1-w_{n})\,n=\frac{n}{1+\lambda T_{n}}\longrightarrow0 .
	\end{equation}
\end{lemma}

\begin{proof}
	By $t_{m}=S_{m}^{2}/d_{m}^{2}$ and the norm ratio
	$d_{n-1}^{2}/d_{n}^{2}=1/(n(N-n+1)p(1-p))$ of Table~\ref{tab:kraw},
	$t_{n}/t_{n-1}=(S_{n}/S_{n-1})^{2}\,d_{n-1}^{2}/d_{n}^{2}$. The forward-shift
	\eqref{eq:fwdK} gives $S_{n}/S_{n-1}=\tfrac{n}{n-j}\,\zeta^{[N-j]}_{n-j}(\alpha)$, and by
	the leading order \eqref{eq:zetacrude} at $w=\alpha$, length $N-j$,
	$\zeta^{[N-j]}_{n-j}(\alpha)=-(1-p)(n-j)(1+o(1))$, so $(S_{n}/S_{n-1})^{2}=(1-p)^{2}n^{2}(1+o(1))$.
	Therefore
	\[
		\frac{t_{n}}{t_{n-1}}=\frac{(1-p)^{2}n^{2}(1+o(1))}{n(N-n+1)p(1-p)}
		=\frac{(1-p)\,n}{(N-n+1)\,p}(1+o(1))\longrightarrow\frac{(1-p)r}{(1-r)p}=L_{t},
	\]
	which is \eqref{eq:Lt2}; only the pointwise leading order \eqref{eq:zetacrude} is used.
	This local ratio is not used to infer growth from an $N$-dependent initial value.
	Instead, Lemma~\ref{lem:boundary} gives
	$t_{n-1,N}=\exp(N\Psi_p(r)+o(N))$, with $\Psi_p(r)>0$. Hence
	$T_{n,N}\ge t_{n-1,N}$ grows exponentially, so $T_{n,N}/n\to\infty$ and
	$(1-w_{n,N})n=n/(1+\lambda T_{n,N})\to0$.

	It remains to determine the mass ratio. Set
	$\phi_{m,N}=T_{m-1,N}/t_{m-1,N}$. From
	$T_{n}=T_{n-1}+t_{n-1}$,
	the following is an identity with $N$ fixed:
	\begin{equation}\label{eq:phirec}
		\phi_{m+1,N}=\frac{1}{c_{m,N}}\,\phi_{m,N}+\frac{1}{c_{m,N}},
		\quad c_{m,N}=\frac{t_{m,N}}{t_{m-1,N}}.
	\end{equation}
	Choose $\eta>0$ with $p<r-2\eta$ and $a_N=\lfloor(r-\eta)N\rfloor$.  The ratio
	calculation above, applied along arbitrary sequences in
	$[r-\eta,r+o(1)]$, makes $1/c_{m,N}\le q<1$ uniformly there.  By
	Lemma~\ref{lem:boundary}, $\phi_{a_N,N}=\exp(o(N))$, so iteration of
	\eqref{eq:phirec} across the macroscopic window kills the boundary term and makes
	$\phi_{m,N}$ uniformly bounded on its terminal half.  On the last
	$\ell_N=\lfloor\sqrt N\rfloor$ indices, $1/c_{m,N}\to1/L_t$ uniformly.  Applying
	Lemma~\ref{lem:triangular} on that terminal window gives
	$\phi_{n,N}\to(1/L_t)/(1-1/L_t)=1/(L_t-1)$. Therefore
	$T_{n-1,N}/T_{n,N}=\phi_{n,N}/(\phi_{n,N}+1)\to1/L_t=\varrho$, proving
	\eqref{eq:massratio} without treating \eqref{eq:phirec} as a diagonal recursion.
\end{proof}

\begin{lemma}[Top-bracket limit]\label{lem:topbracket}
	For each fixed $z\in\C\setminus\Nzero$,
	\begin{equation}\label{eq:topbracket}
		n\,\theta_{n}(z)\,Q_{n,n-1}(z)\longrightarrow
		\beta(z)=-\,\frac{r}{r-p}\,(z-\alpha)\,\varphi_{K}(z),
		\quad n\to\infty .
	\end{equation}
\end{lemma}

\begin{proof}
	Write $P_{s}(z)=\theta_{s}(z)\Kr{s}(z)\to\varphi_{K}(z)$ by \eqref{eq:classicalMHK}. From
	$Q_{n,n-1}(z)=\Kr{n}(z)-\tfrac{S_{n}}{S_{n-1}}\Kr{n-1}(z)$ and
	$\theta_{n}(z)/\theta_{n-1}(z)=-1/((1-p)(n-1-z))$ of \eqref{eq:kapparatio},
	\begin{multline}
	    \label{eq:topsplit}
		n\theta_{n}(z)Q_{n,n-1}(z)
		=n P_{n}(z)-\frac{S_{n}}{S_{n-1}}\,\frac{n\,\theta_{n}(z)}{\theta_{n-1}(z)}P_{n-1}(z)\\
		=n P_{n}(z)+\frac{S_{n}}{S_{n-1}}\,\frac{n}{(1-p)(n-1-z)}\,P_{n-1}(z) .
	\end{multline}
	By the forward-shift \eqref{eq:fwdK},
	$S_{n}/S_{n-1}=\tfrac{[n]_{j}}{[n-1]_{j}}\,\zeta^{[N-j]}_{n-j}(\alpha)$ with
	$[n]_{j}/[n-1]_{j}=n/(n-j)$, and by Lemma~\ref{lem:ratioexp} at $w=\alpha$ (length
	$N-j$, ratio $=n/N\to r$) one has
	\begin{equation}\label{eq:j-cancel}
		\frac{S_{n}}{S_{n-1}}
		=\frac{n}{n-j}\bigl[-(1-p)((n-j)-\gamma(\alpha))+o(1)\bigr]
		=-(1-p)\bigl(n-\gamma(\alpha)\bigr)+o(1),
	\end{equation}
	which is the same expansion as in the case $j=0$ (where the factor $n/(n-j)$ equals
	$1$ and $\zeta^{[N-j]}_{n-j}$ is replaced by $\zeta^{[N]}_{n}$, with the same limit
	$\gamma(\alpha)$): the $j$-dependent factors $[n]_{j}/[n-1]_{j}$ and the shift
	$N-j$ cancel exactly at the leading and second order, which is the content of the
	independence of $j$ in Theorem~\ref{thm:mainMH}. Hence
	\[
		\frac{S_{n}}{S_{n-1}}\,\frac{n}{(1-p)(n-1-z)}
		=-\frac{n\,(n-\gamma(\alpha))}{n-1-z}+o(1) .
	\]
	By Lemma~\ref{lem:ratioexp} at $w=z$, $P_{n}(z)/P_{n-1}(z)
	=(\theta_{n}/\theta_{n-1})\zeta_{n}(z)=\dfrac{n-\gamma(z)}{n-1-z}+o(1/n)
	=1+\dfrac{1+z-\gamma(z)}{n}+o(1/n)$, hence
	$P_{n-1}(z)=P_{n}(z)\bigl(1-\tfrac{1+z-\gamma(z)}{n}+o(1/n)\bigr)$. Substituting both into
	\eqref{eq:topsplit},
	\[
		n\theta_{n}(z)Q_{n,n-1}(z)
		=P_{n}(z)\Bigl[n-\frac{n(n-\gamma(\alpha))}{n-1-z}
		\Bigl(1-\tfrac{1+z-\gamma(z)}{n}\Bigr)\Bigr]+o(1) .
	\]
	Expanding $\dfrac{n(n-\gamma(\alpha))}{n-1-z}=n+(1+z-\gamma(\alpha))+o(1)$ and multiplying
	by $\bigl(1-\tfrac{1+z-\gamma(z)}{n}\bigr)$ gives $n+\gamma(z)-\gamma(\alpha)+o(1)$, so the
	bracket tends to $-(\gamma(z)-\gamma(\alpha))=-\tfrac{r}{r-p}(z-\alpha)$ by
	\eqref{eq:rhogammadef}. Therefore $n\theta_{n}(z)Q_{n,n-1}(z)\to
	-\tfrac{r}{r-p}(z-\alpha)\varphi_{K}(z)=\beta(z)$.
	The same calculation holds for every limiting ratio $x\in(p,1)$, with $r$ replaced by
	$x$; a subsequence argument makes the convergence and boundedness uniform when $n/N$
	ranges over a compact subinterval of $(p,1)$.
\end{proof}

\begin{theorem}\label{thm:mainMH}
	Let $\{\Sob{n}\}$ be the Krawtchouk--Sobolev family for a single left mass
	($\lambda>0$, $\alpha<0$, $j\ge1$, all fixed), as $n,N\to\infty$ with $n/N\to r$ and
	$0<p<r<1$. Then, for each fixed $z\in\C\setminus\Nzero$,
	\begin{equation}\label{eq:mainMH}
		\lim_{\substack{n,N\to\infty\\ n/N\to r}}
		\frac{n}{\kappa_{n}\,\Gamma(n-z)}\,\Sob{n}(z)
		=-C_{\ast}\,(z-\alpha)\,\varphi_{K}(z)=G(z),
		\quad C_{\ast}=\frac{r^{2}(1-p)}{(r-p)^{2}} .
	\end{equation}
	The limit is independent of the mass strength $\lambda>0$ and of the difference order
	$j\ge1$ (each held fixed).
\end{theorem}

\begin{proof}
	Fix $z\in\C\setminus\Nzero$. All quantities in \eqref{eq:Yrec} are first regarded as
	triangular arrays with $N$ fixed. Its forcing term at $m=n(N)$ tends to $\beta(z)$ by
	Lemma~\ref{lem:topbracket}. Its multiplier is
	\[
		\nu_{m,N}=\frac{T_{m-1,N}}{T_{m,N}}\frac{S_{m,N}}{S_{m-1,N}}
		\frac{m\theta^{[N]}_{m}(z)}{(m-1)\theta^{[N]}_{m-1}(z)}.
	\]
	By \eqref{eq:zetacrude}, \eqref{eq:kapparatio}, and
	Lemma~\ref{lem:massratio} applied along any sequence $m/N\to x\in(p,1)$,
	\[
		\frac{S_{m,N}}{S_{m-1,N}}
		\frac{m\theta^{[N]}_{m}(z)}{(m-1)\theta^{[N]}_{m-1}(z)}\longrightarrow1,
		\quad \nu_{m,N}\longrightarrow\varrho(x)=\frac{p(1-x)}{x(1-p)}.
	\]

	Choose $\eta>0$ with $p<r-2\eta$ and put $a_N=\lfloor(r-\eta)N\rfloor$.
	The preceding sequential limits imply that on $a_N<m\le n(N)$ the multipliers are
	uniformly bounded by some $q<1$, while the forcing terms are uniformly bounded.
	Lemma~\ref{lem:boundary} gives $|Y_{a_N,N}(z)|=\exp(o(N))$; direct iteration of
	\eqref{eq:Yrec} therefore kills the boundary contribution
	$q^{n-a_N}Y_{a_N,N}$ and makes $Y_{m,N}$ uniformly bounded on the terminal half of
	this macroscopic window.

	On the last $\ell_N=\lfloor\sqrt N\rfloor$ indices, $m/N\to r$ uniformly. Thus the
	forcing tends uniformly to $\beta(z)$, the multiplier tends uniformly to $\varrho$, and
	$q^{\ell_N}|Y_{n-\ell_N,N}(z)|\to0$. Lemma~\ref{lem:triangular} now gives
	\[
		Y_{n,N}(z)\longrightarrow\frac{\beta(z)}{1-\varrho}
		=-\frac{r}{r-p}\cdot\frac{1}{1-\varrho}\,(z-\alpha)\varphi_{K}(z).
	\]
	Since $\frac{r}{r-p}\frac{1}{1-\varrho}=\frac{r^{2}(1-p)}{(r-p)^{2}}=C_{\ast}$,
	we obtain $Y_{n,N}(z)\to-C_{\ast}(z-\alpha)\varphi_{K}(z)=G(z)$.
	Finally, applying $n\theta_{n}(z)$ to \eqref{eq:SobXi},
	\[
		\frac{n}{\kappa_{n}\Gamma(n-z)}\Sob{n}(z)
		=(1-w_{n})\,n\theta_{n}(z)\Kr{n}(z)+w_{n}\,Y_{n}(z) ;
	\]
	here $n\theta_{n}(z)\Kr{n}(z)=nP_{n}(z)$ with $P_{n}(z)\to\varphi_{K}(z)$ bounded, and
	$(1-w_{n})n\to0$ by Lemma~\ref{lem:massratio}, so the first term vanishes; since
	$w_{n}\to1$, the left side tends to $G(z)$, which is \eqref{eq:mainMH}. The limit contains
	neither $\lambda$ (the fixed-$N$ recursion \eqref{eq:Yrec} and the ratio $T_{n-1}/T_{n}$ being
	$\lambda$-free, and $w_{n}\to1$) nor $j$, which enters only through $\gamma(\alpha)$ in
	Lemma~\ref{lem:topbracket} via $S_{n}/S_{n-1}$, and there the $j$-dependent shift cancels.
\end{proof}

\subsection{The Uvarov case $j=0$}
\label{subsec:uvarov}

The proof of Theorem~\ref{thm:mainMH} never uses $j\ge1$; the difference order enters only
through the factor $n/(n-j)$ and the length $N-j$ in the ratio $S_{n}/S_{n-1}$, and both
reduce, for $j=0$, to the identity and to the length $N$. We record this explicitly.

\begin{corollary}[Uvarov case]\label{cor:uvarov}
	Let $\{\Sob{n}\}_{0\le n\le N}$ be the monic orthogonal polynomials for the Uvarov
	modification
	\begin{equation}\label{eq:uvarov-ip}
		\langle f,g\rangle_{\lambda}
		=\sum_{x=0}^{N} f(x)g(x)\,\rho(x)+\lambda\,f(\alpha)\,g(\alpha),
	\end{equation}
	with $\lambda>0$ and $\alpha<0$ fixed, and let $n,N\to\infty$ with $n/N\to r$ and
	$0<p<r<1$. Then, for each fixed $z\in\C\setminus\Nzero$,
	\begin{equation}\label{eq:uvarov-limit}
		\lim_{\substack{n,N\to\infty\\ n/N\to r}}
		\frac{n}{\kappa_{n}\,\Gamma(n-z)}\,\Sob{n}(z)
		=-C_{\ast}\,(z-\alpha)\,\varphi_{K}(z),
	\end{equation}
	with the same constant $C_{\ast}=r^{2}(1-p)/(r-p)^{2}$; in particular the limit is
	again independent of $\lambda>0$.
\end{corollary}

\begin{proof}
	For $j=0$ the forward difference reduces to the identity, so \eqref{eq:uvarov-ip} is
	\eqref{eq:sobinner}, Theorem~\ref{thm:connection} holds verbatim with
	$\dif^{0}=\mathrm{Id}$, and the quantities of
	\eqref{eq:tm}--\eqref{eq:Xidef} read
	$S_{m}=\Kr{m}(\alpha)$, $\ff{m}{0}=1$, $t_{m}=\Kr{m}(\alpha)^{2}/d_{m}^{2}$,
	$T_{n}=\Ker_{n-1}(\alpha,\alpha)$. Lemma~\ref{lem:xirec} applies for $n\ge1$: for
	$n=1=j+1$ one has $T_{0}=0$, $\mu_{1}=0$, and
	$\Xi_{1}=\Kr{1}-S_{1}H_{1}=\Kr{1}-S_{1}\frac{S_{0}\Kr{0}}{d_{0}^{2}T_{1}}
	=\Kr{1}-\frac{S_{1}}{S_{0}}\Kr{0}=Q_{1,0}$ by \eqref{eq:tm}--\eqref{eq:Tn} with
	$T_{1}=t_{0}=S_{0}^{2}/d_{0}^{2}$; the case $n\ge j+2$ is the proof of
	Lemma~\ref{lem:xirec} unchanged. Lemma~\ref{lem:boundary} applies for $j=0$: in
	\eqref{eq:tail-start}--\eqref{eq:rate-tail} one has $s=k$, $M=N$, $\ff{k}{0}=1$, and
	the estimates \eqref{eq:phi-bound}--\eqref{eq:tail-bound} and
	\eqref{eq:S-asympt}--\eqref{eq:rate-active} are unchanged, as are
	\eqref{eq:z-alpha-comparison}--\eqref{eq:alpha-S-comparison} and
	\eqref{eq:theta-S-cancel}. In particular the extreme case $j=0$, $k/N=p$ is covered by
	\eqref{eq:phi-bound}: there $s/(Mp)=1$ exactly, the geometric majorant is merely
	replaced by the elementary bound $\sum_{i=0}^{s}(s/(Mp))^{i}\le s+1$, and each
	summand of \eqref{eq:phi-series} is still at most $C(\alpha)\,i^{-\alpha-1}$, so
	$\log\Phi=O(\log N)$ as before. Finally, in Lemma~\ref{lem:massratio} and
	Lemma~\ref{lem:topbracket} the ratio $S_{n}/S_{n-1}$ is replaced by
	$\zeta^{[N]}_{n}(\alpha)$, i.e.\ $n/(n-j)=1$ and $N-j=N$: by
	Lemma~\ref{lem:ratioexp} at length $M=N$ it satisfies the same expansion
	$\zeta^{[N]}_{n}(\alpha)=-(1-p)\bigl(n-\gamma(\alpha)\bigr)+o(1)$, so both lemmas and,
	with them, Theorem~\ref{thm:mainMH} give \eqref{eq:uvarov-limit}.
\end{proof}

\begin{remark}[The scale $n$ and the role of $\lambda$]\label{rem:scale}
	The identity \eqref{eq:SobXi} is
	\begin{equation}\label{eq:scale-id}
		\frac{n}{\kappa_{n}\Gamma(n-z)}\Sob{n}(z)
		=(1-w_{n})\,n\theta_{n}(z)\Kr{n}(z)+w_{n}\,Y_{n}(z),
	\end{equation}
	and in the passage to the diagonal limit the first term disappears: $w_{n}\to1$ and
	$(1-w_{n})n\to0$ by Lemma~\ref{lem:massratio}, while
	$n\theta_{n}(z)\Kr{n}(z)=nP_{n}(z)$ grows like $n\varphi_{K}(z)$. Two consequences
	deserve to be stated. First, the limit $G$ is carried entirely by the scaled Sobolev
	correction $Y_{n}$, not by the classical term; the rank-one perturbation, in this
	scale, asymptotically dominates the connection formula. Thus the word
	``Mehler--Heine'' here does not describe a smooth deformation of the classical case
	at the level of the limit: the classical term is asymptotically negligible under the normalisation
	\eqref{eq:thetaK} times $1/n$, and the mechanism is the exponential growth of the total
	mass $T_{n,N}$ of the difference kernel (Lemma~\ref{lem:massratio}), which forces
	$w_{n}\to1$. Second, the normalisation \eqref{eq:thetaK} is \emph{$n$ times} the
	classical one: for $\lambda=0$ one has $\Sob{n}=\Kr{n}$ and the quantity
	\eqref{eq:thetaK} of the left-hand side of \eqref{eq:scale-id} behaves like
	$n\varphi_{K}(z)$, which diverges. The theorem therefore establishes independence of
	$\lambda$ for each fixed $\lambda>0$; it does not assert continuity at $\lambda=0$,
	and the limit $G$ cannot be obtained by setting $\lambda=0$ in the normalised
	family. The values $\lambda=\lambda_{n}$ tending to $0$ belong to the regime of
	rescaled masses: the natural scaling is $\lambda_{n}T_{n,N}\asymp1$, i.e.
	$\lambda_{n}\asymp e^{-N\Psi_{p}(r)}$, for which the two terms of
	\eqref{eq:scale-id} would stay of comparable size; we do not treat it here (see the
	conclusions).
\end{remark}

\section{Numerical illustration}
\label{sec:numerics}

We illustrate Theorem~\ref{thm:mainMH} for a single left exterior mass with parameters
$p=0.3$, $r=3/5$ (so that $n=3t$, $N=5t$ makes $r=n/N$ exact), $\alpha=-2.5$, $j=2$ and
$\lambda=5$, for which $C_{\ast}=r^{2}(1-p)/(r-p)^{2}=2.8$. The scaled sequence is
\begin{equation}\label{eq:Fn}
	F_{n}(z)=\frac{n}{\kappa_{n}\,\Gamma(n-z)}\,\Sob{n}(z),
\end{equation}
and for every fixed $z\in\C\setminus\Nzero$ the limit of Theorem~\ref{thm:mainMH} reads
\begin{equation}\label{eq:Fn-limit}
	F_{n}(z)\longrightarrow-C_{\ast}(z-\alpha)\,\varphi_{K}(z),
\end{equation}
whereas at the lattice nodes $z\in\Nzero$ the theorem does not apply.
Since $\deg\Sob{n}=n$ reaches $135$ in the portraits and $300$ on the real axis, where
$\Gamma(n-z)$ and $\Sob{n}$ overflow double precision while their ratio stays of order one,
all quantities are evaluated with $80$ decimal digits in multiprecision arithmetic
(\texttt{mpmath}): the classical polynomials from the recurrence \eqref{eq:ttrr},
cross-checked at $50$ digits against
\eqref{eq:Khyp} and the residual of \eqref{eq:sodeK}, and $\Sob{n}$ from
\eqref{eq:connection}. The self-checking script reproducing every number of
this section is available from the corresponding author and will be provided to the
editors and referees during the evaluation on request, and it will accompany the
archived data set of this manuscript.

Figure~\ref{fig:portraits} shows $\log_{10}|F_{n}(z)|$ over $\mathrm{Re}\,z\in[-4,4.5]$,
$\mathrm{Im}\,z\in[-1.2,1.2]$ for $n=45,90,135$ (panels (a)--(c)) and the limit
$-C_{\ast}(z-\alpha)\varphi_{K}(z)$ (panel (d)). As the degree grows the finite-degree
portraits converge to the limiting profile at each fixed point off the lattice; the square
marks the exterior mass point $\alpha=-2.5$, and the circles mark the classical lattice
$0,1,2,\dots$ of the reciprocal-Gamma profile. The convergence described by
Theorem~\ref{thm:mainMH} concerns each fixed $z\notin\Nzero$; the behaviour at the
lattice nodes is visible numerically but is not covered by the theorem.

\begin{figure}[ht]
	\centering
	\includegraphics[width=0.74\linewidth]{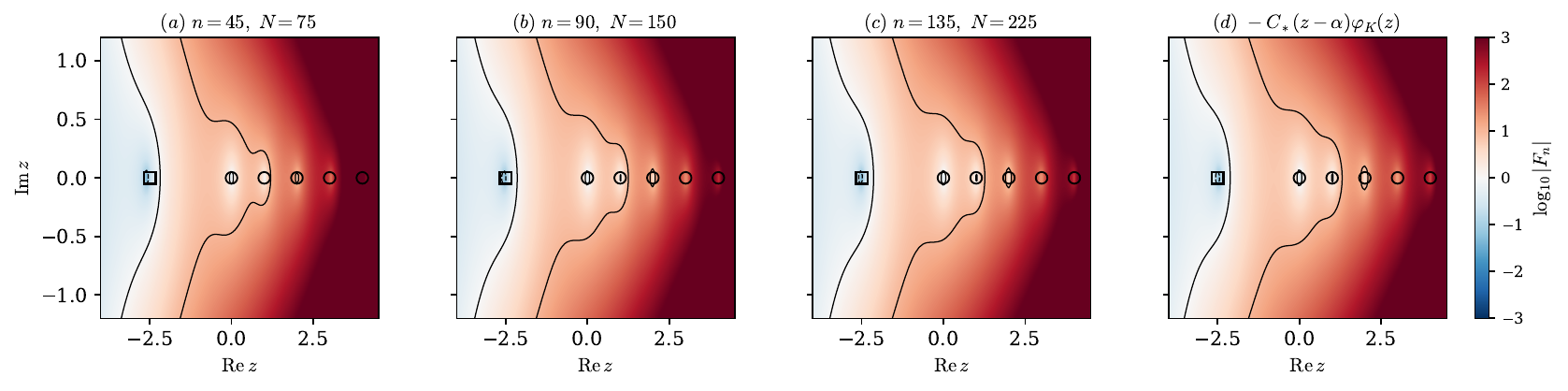}
	\caption{Complex-plane log-modulus portraits of $F_{n}(z)$ of \eqref{eq:Fn}
		($p=0.3$, $r=3/5$, $\alpha=-2.5$, $j=2$, $\lambda=5$). Panels (a)--(c): degrees
		$n=45,90,135$ ($N=75,150,225$); panel (d): the limit
		$-C_{\ast}(z-\alpha)\varphi_{K}(z)$. The square marks the exterior mass point
		$\alpha$, and the open circles mark the lattice $\Nzero$ of the
		reciprocal-Gamma profile.}
	\label{fig:portraits}
\end{figure}

For real $z$ the scaled quantity $F_{n}(z)$ is real. Figure~\ref{fig:realaxis} displays
$F_{n}(z)$ for $n=60,120,180,240,300$ on $z\in[-4,2.6]$, together with
$-C_{\ast}(z-\alpha)\varphi_{K}(z)$; as $n$ increases the coloured curves track the black
limit ever more closely at each fixed $z$, in agreement with the pointwise limit
\eqref{eq:Fn-limit} of Theorem~\ref{thm:mainMH}. At the lattice nodes the theorem does not
apply: the values displayed there are computed directly and shown for completeness,
but they are not covered by the theorem.

\begin{figure}[ht]
	\centering
	\includegraphics[width=0.5\linewidth]{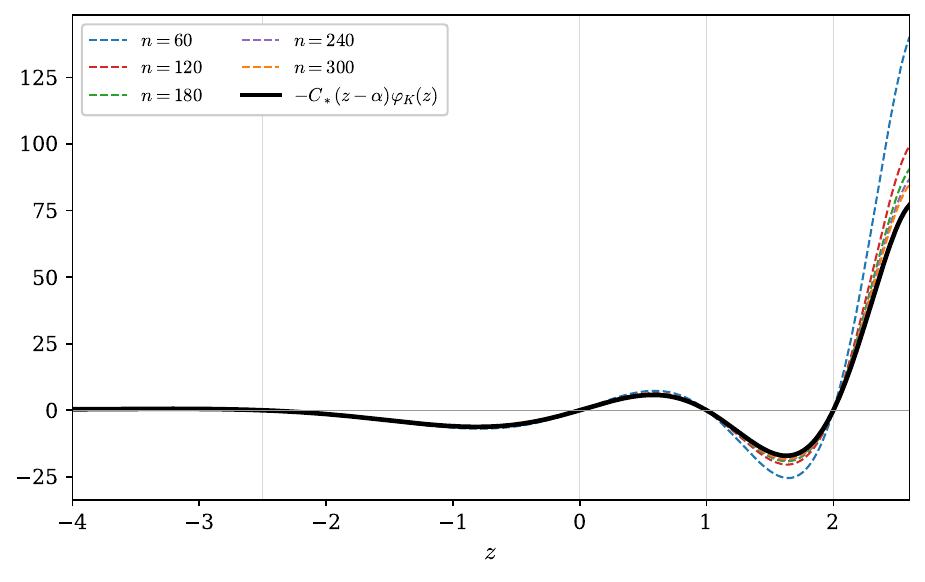}
	\caption{Real-axis convergence of $F_{n}(z)$ (dashed) for $n=60,120,180,240,300$
		to the limit $-C_{\ast}(z-\alpha)\varphi_{K}(z)$ (thick black), parameters as in
		Figure~\ref{fig:portraits}. Vertical guides mark $\alpha=-2.5$ and the nodes
		$0,1,2$.}
	\label{fig:realaxis}
\end{figure}

\FloatBarrier

%%%%%%%%%%%%%%%%%%%%%%%%%%%%%%%%%%%%%%%%%%%%%%%%%%%%%%%%%%%%%%%%%%%%%%%%%%%%%%%%%%%%
\section{Conclusions}
\label{sec:conclu}
 
We have established the large-degree behaviour of the monic Krawtchouk--Sobolev type
polynomials $\Sob{n}$ associated with a single exterior mass point $\alpha<0$, placed off
the support $\{0,1,\dots,N\}$ and acting through a fixed forward difference of order $j$.
In the joint limit $n,N\to\infty$, $n/N\to r$, with $0<p<r<1$ (the reciprocal-Gamma regime
of Dominici's classification), Theorem~\ref{thm:mainMH} proves that, for every fixed
$z\in\C\setminus\Nzero$,
\begin{equation*}
\frac{n}{\kappa_{n}\,\Gamma(n-z)}\,\Sob{n}(z)
\longrightarrow
-C_{\ast}\,(z-\alpha)\,\varphi_{K}(z),
\quad
C_{\ast}=\frac{r^{2}(1-p)}{(r-p)^{2}}.
\end{equation*}
The exterior perturbation thus remains visible in the limiting profile through the single
linear factor $z-\alpha$, which recovers the position of the mass point exactly, whereas
neither its fixed strength $\lambda>0$ nor the fixed difference order $j\ge1$ leaves any
trace. The same formula, with the same constant $C_{\ast}$, holds for the pure Uvarov
modification $j=0$ (Corollary~\ref{cor:uvarov}).
 
The derivation also identifies the mechanism behind this universality through elementary,
explicitly computed objects. The rank-one connection formula
(Theorem~\ref{thm:connection}) and the reproducing-kernel recursion reduce the Sobolev
correction to an exact affine recursion at fixed $N$, whose limiting multiplier
$\varrho=p(1-r)/(r(1-p))$ satisfies $\varrho<1$ precisely when $r>p$. This permits the
triangular contraction of Lemma~\ref{lem:triangular}, which, combined with the pointwise
two-term expansion of the consecutive-index Krawtchouk ratio (Lemma~\ref{lem:ratioexp})
read at $z$ and at $\alpha$, produces both the factor $z-\alpha$ and the constant. The
latter is transparent, factoring as $C_{\ast}=\frac{r}{r-p}\cdot\frac{1}{1-\varrho}$, the
product of the turn-point geometry and the fixed point of the contraction; the exact
cancellation of the $j$-dependent terms in the ratio $S_{n}/S_{n-1}$ is what makes the
limit the same for all difference orders, $j=0$ included. Methodologically the route is
pointwise and self-contained: its only nonelementary input is Dominici's pointwise limit
\cite[Cor.~3(ii)]{Dominici2020}, used exactly as stated with no uniformity or rate, the
finite second-order information being supplied by the three-term recurrence.
 
A defining feature is the scale at which the perturbation is observed. Under the
normalisation of the theorem, which is $n$ times the classical Mehler--Heine one, the
classical Krawtchouk contribution becomes asymptotically negligible and the normalised
Sobolev correction carries the entire limit, driven by the exponential growth of the total
mass $T_{n,N}$ of the difference kernel. Accordingly, the independence of $\lambda$ holds
for each fixed $\lambda>0$ but does not imply continuity at $\lambda=0$, where the same
normalisation yields the divergent quantity $n\,\varphi_{K}(z)$ (Remark~\ref{rem:scale}).
The persistence of the exterior mass through $z-\alpha$ is the reciprocal-Gamma Krawtchouk
counterpart of the phenomenon established for the discrete Charlier and Meixner
Sobolev-type families in \cite{SoriaMichel2026}, and contrasts with the on-support
Krawtchouk--Sobolev family of \cite{Huertas2022}: the off-support placement $\alpha<0$ is
what lets the location of the mass survive in the local large-degree limit. The
multiprecision computations of Section~\ref{sec:numerics} confirm the pointwise
convergence and the exact value $C_{\ast}=2.8$ for $p=0.3$, $r=3/5$, $\alpha=-2.5$, $j=2$,
$\lambda=5$, while delimiting its scope: convergence is pointwise for fixed
$z\in\C\setminus\Nzero$, with no claim at the lattice nodes $z\in\Nzero$ or about locally
uniform convergence.
 
Several questions remain open. The critical transition $r\to p$, where $\varrho\to1$ and
the triangular argument degenerates, and the complementary regime $r<p$ require a separate
analysis. Varying masses $\lambda=\lambda_{n}$ are also natural, with critical scaling
$\lambda_{n}T_{n,N}\asymp1$, i.e.\ $\lambda_{n}\asymp e^{-N\Psi_{p}(r)}$, for which the
classical and Sobolev contributions could coexist in the limit (Remark~\ref{rem:scale}).
A difference order $j=j_{n}$ growing with $n$, the behaviour at the lattice nodes, several
simultaneous exterior masses, and stronger locally uniform convergence are further
avenues we intend to pursue.

%%%%%%%%%%%%%%%%%%%%%%%%%%%%%%%%%%%%%%%%%%%%%%%%%%%%%%%%%%%%%%%%%%%%%%%%%%%%%%%%%%%%
\section*{Acknowledgements}

The authors would like to thank the Department of Quantitative Methods at Universidad Loyola Andalusia for providing an excellent research environment and institutional support during the development of this work.

\section*{Declarations}

\subsection*{Ethical Approval}
Not applicable.

\subsection*{Consent to Participate}
Not applicable.

\subsection*{Consent to Publish}
Not applicable.

\subsection*{Data Availability Statement}
No external datasets were used in this study.

\subsection*{Author Contributions}
Conceptualization, A.S.-L. and J.-M.; methodology, A.S.-L. and J.-M.; software, A.S.-L. and J.-M.; validation, A.S.-L. and J.-M.; formal analysis, A.S.-L. and J.-M.; investigation, A.S.-L. and J.-M.; resources, A.S.-L. and J.-M.; data curation, A.S.-L. and J.-M.; writing--original draft preparation, A.S.-L. and J.-M.; writing--review and editing, A.S.-L. and J.-M.; visualization, A.S.-L. and J.-M.; supervision, A.S.-L. and J.-M.; project administration, A.S.-L. and J.-M. Both authors have read and agreed to the published version of the manuscript.

\subsection*{Funding}
This research received no external funding.

\subsection*{Competing Interests}
The authors declare that they have no competing interests.

%%%%%%%%%%%%%%%%%%%%%%%%%%%%%%%%%%%%%%%%%%%%%%%%%%%%%%%%%%%%%%%%%%%%%%%%%%%%%%%%%%%%

\end{document}